\documentclass[final,1p,times]{elsarticle}
\usepackage[centertags]{amsmath}
\usepackage{amssymb}
\usepackage{amsthm}
\newtheorem{thm}{Theorem}

\newtheorem{prop}[thm]{Proposition}
\newtheorem{cor}[thm]{Corollary}
\newdefinition{rem}{Remark}
\newdefinition{example}{Example}
\newdefinition{defn}{Definition}
\newcommand{\norm}[1]{\left\Vert#1\right\Vert}

\newcommand{\condexp}[1]{\left|}
\newcommand{\Real}{\mathbb R}

\newcommand{\A}{\mathcal{A}}
\newcommand{\punt}{\boldsymbol{.}\: \!}

\newcommand{\tbs}{\boldsymbol{t}}

\newcommand{\ybs}{\boldsymbol{y}}

\newcommand{\lambdabs}{\boldsymbol{\lambda}}
\newcommand{\mmodels}{\boldsymbol{\vdash}}

\newcommand{\ibs}{\boldsymbol{i}}
\newcommand{\mubs}{\boldsymbol{\mu}}

\newcommand{\nubs}{\boldsymbol{\nu}}

\newcommand{\jbs}{\boldsymbol{j}}
\newcommand{\ml}{\mathfrak{m}}

\newcommand{\zbs}{\boldsymbol{z}}

\newcommand{\cop}{\scriptscriptstyle <-1>}

\newcommand{\Tr}{\hbox{\rm Tr}}
\journal{Journal of Multivariate Analysis}

\begin{document}

\begin{frontmatter}



\title{On a symbolic representation of non-central Wishart random matrices with applications}
\author[ed]{Elvira Di Nardo\corref{cor1}}
\ead{elvira.dinardo@unibas.it}
\cortext[cor1]{Corresponding author}
\address[ed]{Dept. Mathematics, Computer Science and Economics,
University of Basilicata, Viale dell'Ateneo
Lucano 10, 85100 Potenza, Italia}
\ead[url]{http://www.unibas.it/utenti/dinardo/}
\begin{abstract}
By using a symbolic method, known in the literature as the classical umbral calculus, the trace of a non-central Wishart random matrix is represented as the convolution of the trace of its central component and of a formal variable involving traces of its non-centrality matrix. Thanks to this representation, the moments of this random matrix are proved to be a Sheffer polynomial sequence, allowing us to recover several properties. The multivariate symbolic method generalizes the employment of Sheffer representation and a closed form formula for computing joint moments and cumulants (also normalized) is given. By using this closed form formula and a combinatorial device, known in the literature as necklace, an efficient algorithm for their computations is set up.  Applications are given to the computation of permanents as well as to the characterization of inherited estimators of cumulants, which turn useful in dealing with minors of non-central Wishart random matrices. An asymptotic approximation of generalized moments involving free probability is proposed.
\end{abstract}


\end{frontmatter}


\section{Introduction}
Let $\{X_{(1)}, \ldots, X_{(n)} \}$ be random row vectors independently drawn from a $p$-variate complex normal
distribution with zero mean ${\boldsymbol 0}$ and full rank Hermitian covariance matrix $\Sigma.$ Let
$\boldsymbol{m}_1, \ldots, \boldsymbol{m}_n$ be row vectors of dimension $p.$
The {\it non-central Wishart distribution} is the distribution of the square random matrix of order $p$
\begin{equation}
W_p(n, \Sigma, M) = \sum_{i=1}^n (X_{(i)} - \boldsymbol{m}_i)^\dag (X_{(i)} - \boldsymbol{m}_i), \qquad \hbox{with} \,\, M = \sum_{i=1}^n \boldsymbol{m}_i^\dag \boldsymbol{m}_i,
\label{(wishnoncentral)}
\end{equation}
denoted by $W(n)$ for simplicity. The matrix $\Omega = \Sigma^{-1} M$ is called the non-centrality matrix and it is usually employed  instead of $M$ to parametrize the Wishart distribution. For $M=0$ the Wishart distribution is said to be {\it central} and denoted by $\widehat{W}(n)$ in the following.

The most popular application of Wishart distributions is within maximum likelihood estimation in connection with the sample covariance matrix. Therefore these distributions are successfully employed in several areas: a good review is given in \cite{Withers}. In parallel to its spread in the applications, also its mathematical properties have received great attention
in the literature, including those involving determinants and eigenvalues. According to the so-called moment method \cite{Tao}, which relies on eigenvalue distribution,
the $i$-th moment of a $p \times p$ random matrix $A$ is $E[\Tr_p(A^i)],$ where $\Tr_p$ denotes the normalized trace. In particular, the computation of joint moments
\begin{equation}
E\left\{ \Tr \left[ W(n) \, H_1 \right]^{i_1} \cdots \Tr \left[ W(n) \, H_m \right]^{i_m} \right\}
\qquad \hbox{with} \quad H_1, \ldots, H_m \in {\mathbb C}^{p \times p}
\label{(joint)}
\end{equation}
is a very general task: for example, when $H_1, \ldots, H_m$ are sparse matrices, equation
(\ref{(joint)}) returns joint moments of entries of $W(n).$ Moments of  Wishart random matrices are employed not only to characterize asymptotic properties of the sample covariance matrix \cite{Maiwald} but also to quantify the performance of many multidimensional signal processing algorithms \cite{Tague}.
For example, within wireless communication theory, most literature focuses on channel capacity modeled by a memoryless linear vector of the form 
\begin{equation}
{\bf y} = B {\bf x} + {\bf n}
\label{(modelchannel)}
\end{equation}
with ${\bf x}$ a $p$-dimensional input vector, ${\bf y}$ a $n$-dimensional output vector and ${\bf n}$ a $n$-dimensional vector representing a circularly symmetric Gaussian noise. The $n \times p$ random matrix $B$ encodes the characteristics of the channel system having $n$ receivers and $p$ transmitters.  Quite often performance measures of (\ref{(modelchannel)})
rely on the product $B^{\dag} B,$ which is a non-central Wishart random matrix when the
entries of $B$ are Gaussian i.i.d. random variables (r.v.'s) \cite{Tulino}. Numerical evaluations of these indexes involve cumbersome integrals, but fortunately, much deeper insights can be obtained using the tools provided by asymptotic random matrix theory, that is to characterize the asymptotic spectrum of $B$ as the number of columns and rows goes to infinity. Usually the Stieltjes transformation is employed as uniquely determines the distribution function of the spectrum. As an alternative, especially when the inversion formula of Stieltjes-Perron could not be applied, the evaluation of normalized moments $\Tr(B^{\dag} B)^i$ together with the study of their convergence provide a tool to characterize the limiting spectrum through the moment convergence theorem. Since random matrices are non-commutative objects, for dimensions greater than $8$ free probability tools can be employed to recognize if two or more channels are independent. Within free probability, the classical notion of independence is replaced by the notion of freeness. Freeness allows us to recover the asymptotic spectrum of the sum of random matrices from the individual asymptotic spectra. In order to recognize the freeness property a great number of joint moments need to be computed, so that efficient algorithms are necessary. 

For finite matrix theory, the computation of joint moments is still object of in-depth analysis due to its complexity. In order to compute (\ref{(joint)}) when $i_1 = \cdots = i_m = 1,$  the theory of representation group together with the differential of the Laplace transform in suitable directions have been employed in \cite{Letac1, Letac2}.  Weighted generating functions on matchings of graphs have been used in \cite{Numata}. In order to 
compute joint moments in terms of joint cumulants, multivariate Bell polynomials are resorted in \cite{Withers} but the resulting expressions are cumbersome to deal with. Via Edgeworth expansion, joint cumulants of Wishart distributions are employed in approximating density or distribution functions of some statistics relied on quadratic forms of normal samples \cite{Kollo1}.

The aim of this paper is twofold: to give a closed form formula to compute joint moments (\ref{(joint)}), thus generalizing the result given in \cite{Letac2}, and to introduce a symbolic method, known in the literature as the classical umbral calculus, as a natural way to deal with moments of Wishart distributions. This method allows us to manage moment sequences without making hypothesis on the existence of an underlying distribution probability. Many results rely on the representation of the trace of a non-central Wishart random matrix as the convolution of the traces of its central component and of a formal variable (or symbol), {\it uncorrelated} with the entries of the central component. This representation allows us to express $\{E[\Tr(W(n)^i)]\}$ as a Sheffer polynomial sequence \cite{Nied}. Then many Sheffer properties are shared: for example the role played by complete homogeneous symmetric polynomials as well as cyclic polynomials is highlighted when referred to the central component, whose moments result to be a sequence of binomial type. When the indeterminate of a Sheffer polynomial sequence is replaced by a suitable r.v., a randomized non-central Wishart distribution  is recovered. The multivariate version of
this symbolic method extends the employment of Sheffer polynomial sequences to the computation of joint moments (\ref{(joint)}). Moreover, in order to take into account the cyclic property of traces, the notion of necklace is fruitfully employed in setting up an efficient symbolic procedure. The algorithm in {\tt Maple 12} is available on demand. Parallelisms with  other techniques proposed for computing (\ref{(joint)})  are given forward.

The usefulness of the symbolic representation of non-central Wishart distributions is not only confined to computational issues. To widening its applicability, further applications are proposed at the end of the paper. The first concerns the computation of permanents of special classes of matrices. The second characterizes spectral polykays which are unbiased estimators of cumulants. Spectral polykays have the advantage to preserve its expansion in terms of power sum symmetric polynomials
when referred to principal sub matrices. Last application allows us an asymptotic approximation of generalized moments relied on some results borrowed from
free probability.

In order to make the paper self-contained, a short introduction on the symbolic method and on its multivariate version is given in Section 2 and Section 4 respectively. Open problems still concern an exact computation of generalized moments as well as the employment of the symbolic method in computing moments of a 
non-central Wishart random matrix inverse, which are also employed within signal processing
\cite{Maiwald}. The difficulty of this task relies on the circumstance that these moments are usually expressed in terms of zonal polynomials whose efficient handling is still an appealing problem from a computational point of view.

\section{Moment symbolic method: the univariate case}
The moment symbolic method we use relies on the classical
umbral calculus introduced by Rota and Taylor in $1994$ \cite{SIAM}. The method
has been developed and refined in a series of papers starting from
\cite{Dinsen, Dinardoeurop} reducing the overall computational apparatus to few fundamental relations.
The key point consists in working with symbols instead of scalar or polynomial sequences.
\par
Let ${\mathbb C}$ be the complex field whose elements are called scalars.
An umbral calculus consists in a generating set $\A=\{\alpha,\beta, \ldots\}$,
called \textit{alphabet}, whose elements are named \textit{umbrae}, a polynomial
ring ${\mathbb C}[\A]$ and a linear functional $E\colon {\mathbb C}[\A]\to {\mathbb C}$ called \textit{evaluation}.
The linear functional $E$ is such that $E[1]=1$ and
\begin{equation}
E[\alpha^{i} \beta^{j} \cdots \gamma^{k}] = E[\alpha
^{i}] \, E[\beta^{j}] \cdots E[\gamma^{k}], \quad
{\hbox{(uncorrelation property)}} \label{(iii)}
\end{equation}
for any set of distinct umbrae in $\A$ and for $i,j,k$ nonnegative integers. In particular, we set $ E[\alpha^i] = a_i$ for $i=0,1,\ldots$ with $a_0=1$
and $\alpha \in \A.$ The scalar $a_i$ is the $i$-th moment of $\alpha$
and we say that the sequence $\{a_i\}$ is \emph{umbrally represented}
by $\alpha$.
Conversely, a sequence of scalars $\{a_i\}$ with $a_0=1$ is represented by an
umbra~$\alpha$ if $E[\alpha^i] = a_i$ for $i=0,1,\ldots$.
Indeed it is always possible to insert in the alphabet new symbols such that
each scalar sequence $\{a_i\}$ corresponds to an element in $\A,$ which is not
necessarily unique. Indeed, two umbrae $\alpha$ and $\gamma$ may represent
the same sequence of moments when $E[\alpha^i] = E[\gamma^i]$ for all nonnegative integers $i.$
In this case, $\alpha$ and $\gamma$ are said to be similar, in symbols
$\alpha \equiv \gamma.$ For umbral polynomials $p, q \in {\mathbb C}[\A]$ a nimbler
equivalence is the umbral equivalence: $p \simeq q$ if and only if $E[p]=E[q].$
Umbral polynomials with disjoint supports \footnote{The support of an umbral polynomial $p \in {\mathbb C}[\A]$ is the set of all umbrae which occur.}
are uncorrelated, that is $E[p \, q] = E[p] \, E[q].$

The alphabet $\A$ contains special symbols representing special sequences.
For example the {\it augmentation umbra} $\varepsilon$
is such that $E[\varepsilon^i]=0$ for all positive integers $i$ and
the {\it unity umbra} $u$ is such that $E[u^i]=1$ for all positive integers $i$.
Further special umbrae will be introduced when necessary.
The formal power series in ${\mathbb C}[\A][[z]]$
\begin{equation}
u + \sum_{i \geq 1} \alpha^i \, \frac{z^i}{i!}
\label{(gf1)}
\end{equation}
is the {\it generating function} (g.f.) of an umbra $\alpha$ and
denoted by $e^{\alpha z}$. Moreover, for any exponential formal power series in ${\mathbb C}[[z]]$
\begin{equation}
f(z) = 1 + \sum_{i \geq 1} a_i \, \frac{z^i}{i!}
\label{(formpow)}
\end{equation}
there exists an umbra $\alpha$ such that $E[e^{\alpha z}]=f(z),$  by extending coefficient-wise
the action of $E,$ cf. \cite{Bernoulli}. The formal power series $f(z)$ in (\ref{(formpow)}) is the g.f. of $\alpha,$ usually denoted by $f(\alpha, z)$ to avoid misunderstandings.  
\par
For distinct (and so uncorrelated) umbrae $\alpha$ and $\gamma,$ we have $f(\alpha+\gamma,z) = f(\alpha,z)f(\gamma,z).$
More in general $f(\alpha' + \cdots + \alpha'',z) = [f(\alpha,z)]^n$ if $\alpha', \ldots, \alpha''$ are $n$ uncorrelated umbrae similar to an umbra $\alpha.$
The summation $\alpha' + \cdots + \alpha''$ is denoted by the auxiliary symbol $n \, \punt \, \alpha$ which is named the \emph{dot product} of the nonnegative integer $n$
and the umbra $\alpha.$ Its moments are \cite{Bernoulli}
\begin{equation}
E[(n \, \punt \, \alpha)^i] = \sum_{\lambda \vdash i} (n)_{l({\lambda})} \, d_{\lambda} \, a_{\lambda}
\label{(sum3)}
\end{equation}
where $\lambda = (1^{r_1}, 2^{r_2}, \ldots)$ is a partition\footnote{Recall that a partition of an integer $i$ is a sequence
$\lambda=(\lambda_1,\lambda_2,\ldots,\lambda_t)$, where
$\lambda_j$ are weakly decreasing integers and $\sum_{j=1}^t
\lambda_j = i$. The integers $\lambda_j$ are named {\it parts} of
$\lambda$. The {\it length} of $\lambda$ is the number of its
parts and will be denoted by $l({\lambda})$. A different
notation is $\lambda=(1^{r_1},2^{r_2},\ldots)$, where $r_j$ is the
number of parts of $\lambda$ equal to $j$ and $r_1 + r_2 + \cdots
= l({\lambda})$. We use the classical notation $\lambda \vdash
i$ with the meaning \lq\lq $\lambda$ is a partition of $i$\rq\rq.} of the integer $i$, $a_{\lambda}=a_1^{r_1} a_2^{r_2} \cdots$ and
$$d_{\lambda} = \frac{i!}{(1!)^{r_1} r_1! (2!)^{r_2} r_2! \cdots}.$$
Let us recall that if $\alpha,\gamma \in \A$ then
\begin{equation}
n \, \punt \, (\alpha + \gamma) \equiv n \, \punt  \, \alpha + n \, \punt \, \gamma \qquad \hbox{and} \qquad (n+m) \, \punt \, \alpha \equiv n \, \punt \, \alpha + m \, \punt  \, \alpha.
\label{(distrib)}
\end{equation}
The integer $n$ in (\ref{(sum3)}) can be replaced by $-1.$ The auxiliary umbra $-1 \, \punt \, \alpha$ is the \emph{inverse} umbra
of $\alpha$ with g.f. $f(-1 \, \punt \, \alpha,z)= f(\alpha,z)^{-1}.$ Also the composition of g.f.'s can be symbolically represented by using the dot-product. The first step is to replace the integer $n$ with an umbra
$\gamma \in \A,$ so that
\begin{equation}
f(\gamma \, \punt \, \alpha, z) = f(\gamma, \log[f(\alpha,z)]),
\label{(dotproduct)}
\end{equation}
which is not yet the composition of $f(\alpha,z)$ and $f(\gamma,z)$ but plays a fundamental role in the symbolic method. Indeed, if in
(\ref{(dotproduct)}) the umbra $\gamma$ is replaced by the \emph{singleton umbra} $\chi$ with g.f. $f(\chi,z)=1+z,$ then $f(\chi \, \punt \, \alpha, z) = 1 +  \log[f(\alpha,z)]$ and the moments of the auxiliary umbra $\varkappa_{\alpha} \equiv \chi \, \punt \, \alpha$ result to be the formal cumulants\footnote{For all nonnegative integers $i,$ cumulants ${\mathfrak C}_i(Y)$ of a r.v. $Y$ have the following properties:  (Homogeneity) ${\mathfrak C}_i(a Y) = a^i {\mathfrak C}_i (Y)$ for $a \in {\mathbb C},$ (Semi-invariance) ${\mathfrak C}_1(Y+a)=a + {\mathfrak C}_1(Y), {\mathfrak C}_i(Y+a) = {\mathfrak C}_i(Y)$ for $i \geq 2,$ (Additivity) ${\mathfrak C}_i(Y_1+Y_2) = {\mathfrak C}_i(Y_1) + {\mathfrak C}_i(Y_2),$ if $Y_1$ and $Y_2$ are independent r.v.'s. Formal cumulants are the coefficients of $\log[f(z)]$ with $f(z)$ a formal power series. They share the same properties of cumulants of a r.v.'s. When numerical values are considered, we need to check that the so-called \lq\lq problem of cumulants\rq\rq admits solution, cf. \cite{Bernoulli} and references therein. } of $\alpha$.  The umbra $\varkappa_{\alpha}$ is the $\alpha$-\emph{cumulant umbra}, cf. \cite{Dinardoeurop}. A special cumulant umbra, employed very often in the rest of the paper, is  $\varkappa_{\bar{u}} \equiv \chi \, \punt \, \bar{u},$ with $\bar{u}$ the boolean unity \cite{Petrullo}. The boolean unity represents the sequence $\{i!\}$ and its g.f. is the ordinary formal power series
\begin{equation}
f(\bar{u},z) = 1 + \sum_{i \geq 1} z^i = \frac{1}{1-z}.
\label{(boolean)}
\end{equation}
 From (\ref{(dotproduct)}) and (\ref{(boolean)}), we have $f(\chi \, \punt \, \bar{u},z) = 1 + \log (1-z)^{-1} = 1 + \sum_{i \geq 1} \frac{z^i}{i}$
so that $E[(\chi \, \punt  \, \bar{u})^i] = (i-1)!$ for all positive integers $i.$

If in $\gamma \, \punt \, \alpha$ we replace the umbra $\gamma$ with the Bell umbra
 \footnote{The Bell umbra represents the sequence of Bell numbers. Its
g.f. is $f(\beta,z)=\exp[e^z-1]$.} $\beta,$ the dot-product $\beta \, \punt \, \alpha$ denotes the
$\alpha$-\emph{partition umbra}. One of its main properties is
\begin{equation}
\beta \, \punt \, (\alpha \dot{+} \gamma) \equiv \beta \, \punt \, \alpha + \beta  \, \punt \, \gamma
\label{(dotsum)}
\end{equation}
where the symbol $\alpha \dot{+} \gamma,$ denotes the \emph{disjoint sum} of $\alpha$ and $\gamma$ and umbrally represents the sequence $\{a_n + g_n\}$ with $\{a_n\}$ and $\{g_n\}$ moments of $\alpha$ and $\gamma$
respectively. More in general, the symbol $\dot{+}_{k=1}^n \alpha_k$ denotes an auxiliary umbra whose $i$-th moment
is the summation of the $i$-th moments of $\alpha_1, \ldots, \alpha_n$ respectively, that is
\begin{equation}
{\mathcal S}_i=E \left[\left( \dot{+}_{k=1}^n \alpha_k \right)^i \right] = a_{i,1} + \cdots + a_{i,n},
\label{(powersum)}
\end{equation}
with $a_{i,j}=E[\alpha_j^i]$ for all nonnegative integers $i$ and $j=1,2,\ldots,n.$
Since ${\mathcal S}_i \simeq \alpha_1^i + \cdots + \alpha^i_n,$ the symbol ${\mathcal S}_i$ is used in analogy
with the $i$-th power sum symmetric polynomial. A special auxiliary umbra employed in the following is $-1 \, \punt \,
\beta \, \punt \, \alpha$ with g.f. $f(-1 \, \punt \, \beta \, \punt \, \alpha,z) = \exp \left[ 1 - f(\alpha,z) \right].$
\par
 Nesting dot-products, one of which is the $\alpha$-\emph{partition umbra}, we get the composition umbra. Indeed the symbol
$\gamma \, \punt \, (\beta \, \punt \, \alpha)$, where $\beta$ is the Bell umbra, has g.f. which is
the composition of $f(\alpha,z)$ and $f(\gamma,z),$ that is $f(\gamma \, \punt \, (\beta \, \punt \, \alpha), z) = f(\gamma, f(\alpha,z)-1).$
Parenthesis can be avoided since $\gamma \, \punt \, (\beta \, \punt \, \alpha) \equiv (\gamma \, \punt \, \beta) \, \punt \, \alpha$.
From (\ref{(sum3)}) the moments are \cite{Dinardoeurop}
\begin{equation}
E[(\gamma \, \punt \, \beta \, \punt \, \alpha)^i] = \sum_{\lambda \vdash i} g_{l({\lambda})} \, d_{\lambda} \, a_{\lambda},
\label{comp}
\end{equation}
with $\{g_{i}\}$ umbrally represented by $\gamma.$ For $f(z)=f(\alpha,z),$ the coefficients of the compositional inverse of $f(z)$ are represented by the symbol
$\alpha^{\scriptsize <-1>}$ such that $\alpha^{\scriptsize <-1>} \, \punt \, \beta \, \punt \, \alpha \equiv \alpha \, \punt \, \beta \, \punt \,
\alpha^{\scriptsize <-1>} \equiv \chi.$ The umbra $\alpha^{\scriptsize <-1>}$ is the
compositional inverse of $\alpha.$
\par
Also polynomial sequences in the indeterminate $x$ can be represented by suitable umbrae,
if  the complex field ${\mathbb C}$ is replaced by ${\mathbb C}[x].$ The uncorrelation property (\ref{(iii)}) becomes
$E[x^i \alpha^k \gamma^m \cdots] = x^i E[\alpha^k] E[\gamma^m] \cdots$
for any set of distinct umbrae in $\A$ and for nonnegative integers $i, k, m, \ldots.$ In ${\mathbb C}[x][\A],$ an umbra
is said to be a scalar umbra when its moments are elements of ${\mathbb C},$ while it is said to be a polynomial umbra
if its moments are polynomials of ${\mathbb C}[x].$ A sequence of polynomials
$q_0, q_1, \ldots \in  {\mathbb C}[x]$ is umbrally represented by a polynomial umbra if $q_0 = 1$ and $q_i$ is of degree $i$ for all nonnegative integers $i.$
A special family of polynomials, including many classical polynomial sequences, is the Sheffer polynomial sequence.
In \cite{Nied}, it has been proved that any Sheffer polynomial sequence is umbrally represented by
the so-called {\it Sheffer polynomial umbra} $\alpha + x \, \punt \, \beta \, \punt \, \gamma^{\cop},$ with $\gamma$
an umbra possessing compositional inverse, that follows if $E[\gamma] \ne 0.$ If $\alpha$ is replaced by the augmentation umbra $\varepsilon$ then the Sheffer polynomial sequence
reduces to its associated sequence $x \, \punt \, \beta \, \punt \, \gamma^{\cop},$ satisfying the binomial property \cite{Niederhausen}.

The polynomial ring ${\mathbb C}[x]$ can be further replaced by ${\mathbb C}[x_1, \ldots,x_p]$ with $x_1, \ldots , x_p$  indeterminates
\cite{Bernoulli}. As example, but also because employed in the following, let us
recall that the complete homogeneous symmetric polynomials $\{h_i\}$
are umbrally represented by the umbral polynomial $x_1 \bar{u}_1 +  \cdots + x_p \bar{u}_p,$ with $\bar{u}_1,  \ldots, \bar{u}_p$ uncorrelated umbrae similar to the boolean unity $\bar{u}.$
Indeed we have
\begin{equation}
h_i(x_1, \ldots, x_p) = \sum_{1 \leq j_1 \leq \cdots \leq j_i \leq p} x_{j_1}  \cdots x_{j_i} = \frac{1}{i!} \, E[(x_1 \bar{u}_1 + \cdots + x_p \bar{u}_p)^i].
\label{(omo)}
\end{equation}
\section{Symbolic representation of Wishart distributions}
According to the moment method, the trace
of $W(n)$ is represented by an umbral polynomial $\mu_1 +  \cdots + \mu_p
\in {\mathbb C}[\A],$ where $\{\mu_1, \ldots, \mu_p\}$ are umbral monomials
having not necessarily disjoint supports and corresponding to the eigenvalues of $W(n).$ For the sake of brevity, we denote the summation $\mu_1 + \cdots + \mu_p$ with $\Tr[W(n)]$ as if it was an element of ${\mathbb C}[\A].$  This notation will be employed in the rest of the paper.

\begin{thm} \label{Wifund1}
If $\hbox{\rm diag}(\theta_1, \ldots, \theta_p) = \Lambda$ is the eigenvalue matrix of the covariance $\Sigma = Q \Lambda Q^\dag,$ with $Q$ the eigenvector matrix,
$\{\bar{u}_1, \ldots, \bar{u}_p \}$ uncorrelated umbrae similar to the boolean unity umbra $\bar{u}$ and $\{\nu_1, \ldots, \nu_p\}$ umbral monomials such that
\begin{equation}
\nu_k = \chi \, \punt \, b_{kk} \punt \, \beta \quad \hbox{\rm for} \,\, k=1,2,\ldots,p,
\label{(umbni)}
\end{equation}
with $b_{11} + \cdots + b_{pp}  = \Tr(\Omega)$ and $\Omega$ the non-centrality matrix $\Omega = \Sigma^{-1} M,$ then
\begin{equation}
\Tr[W(n)] \equiv \sum_{k=1}^p [-1 \, \punt \, \beta \, \punt \, (\nu_k  \theta_k \bar{u}_k) +  n \, \punt \, (\theta_k \bar{u}_k)]
\equiv -1 \, \punt \, \beta \, \punt \, \left( \dot{+}_{k=1}^p \nu_k \theta_k  \bar{u}_k \right) + n \, \punt \, \left(+_{k=1}^p  \theta_k \bar{u}_k
\right),
\label{(Iformachiusa1)}
\end{equation}
where the symbol ${+}_{k=1}^p  \theta_k \bar{u}_k$ denotes $\theta_1 \bar{u}_1 + \cdots + \theta_p \bar{u}_p.$
\end{thm}
\begin{proof}
The moment generating function (m.g.f.) of $W(n)$ is \cite{Muirhead}
\begin{eqnarray}
f(\Tr[W(n)],z) & = &  E\left\{ \exp \left( \Tr[W(n)] \, z \right) \right\} = \frac{\exp \left\{- \Tr[(I_p - z \, \Sigma)^{-1} M \, z]\right\}}{[\det{(I_p - z \, \Sigma)}]^n} \label{(laptrasf)} \nonumber \\
 & = &  \frac{\exp \left\{- \Tr[(I_p - z \, \Sigma)^{-1} \Sigma \Omega \, z]\right\}}{[\det{(I_p - z \, \Sigma)}]^n}. \label{gfomega}
\end{eqnarray}
If $\theta_1, \ldots, \theta_p$ are the eigenvalues of $\Sigma,$ then 
$$\det{(I- z \, \Sigma)}^{-1} = \prod_{i=1}^p (1 - z \, \theta_i)^{-1} \qquad \hbox{and} \qquad [\det{(I - z \, \Sigma)}]^{-n} = f[n \, \punt \, (\theta_1 \bar{u}_1 + \cdots + \theta_p \bar{u}_p),z]$$
from (\ref{(boolean)}). From the first equivalence in (\ref{(distrib)})
$$n \, \punt \, \left( {+}_{k=1}^p  \theta_k \bar{u}_k \right) \equiv n \, \punt \, (\theta_1 \bar{u}_1) + \cdots + n \, \punt  \, (\theta_p \bar{u}_p)$$
as given in (\ref{(Iformachiusa1)}). Moreover, in (\ref{gfomega}) if $\Sigma = Q \Lambda Q^\dag,$ then $\Tr[(I_p - z \, \Sigma)^{-1}  \Sigma \, \Omega \, z] = 
\Tr[Q (I - z \, \Lambda)^{-1} \, Q^\dag \, Q \, \Lambda \, Q^\dag \, \Omega \, z] =
\Tr[(I - z \, \Lambda)^{-1} \, \Lambda \, B \, z],$ with $B=Q^\dag \Omega  Q.$ Therefore we have
\begin{equation}
\Tr[(I - z \, \Lambda)^{-1} \, \Lambda \, B \, z] = \sum_{k=1}^p \frac{b_{kk} \theta_k \, z}{(1 - z \, \theta_k)}.
\label{(esponentebis)}
\end{equation}
The $i$-th coefficient of the formal power series in (\ref{(esponentebis)}) is
$i! \, (b_{11} \, \theta_1^{i} + \cdots + b_{pp} \, \theta_p^{i}).$ Since
$i! \, \theta_k^{i} = E[(\theta_k \bar{u}_k)]^i$ and $b_{kk}=E[(\chi \, \punt \, b_{kk} \, \punt \, \beta)^i],$ for all nonnegative integers $i,$ then
\begin{equation}
i! \, \mathcal{S}_i = i! \, (b_{11} \, \theta_1^{i} + \cdots + b_{pp} \, \theta_p^{i}) = E \left[ \left(
\dot{+}_{k=1}^p \nu_k \theta_k  \bar{u}_k \right)^i \right],
\label{(ccc11)}
\end{equation}
with $\nu_k$ given in (\ref{(umbni)})  and $\mathcal{S}_i $ in (\ref{(powersum)}). Note that $\mathcal{S}_i=\Tr(B \Lambda^i)=\Tr(\Omega \Sigma^i).$ Since \cite{Dinardoeurop}
$$-1 \, \punt \, \beta \, \punt \, \left( \dot{+}_{k=1}^p \nu_k \theta_k  \bar{u}_k \right) \equiv -1 \, \punt \, \beta \, \punt  \,(\nu_1 \theta_1 \bar{u}_1) + \cdots + -1 \, \punt \, \beta \, \punt \, (\nu_p \theta_p  \bar{u}_p),$$
both the contributions in (\ref{(Iformachiusa1)}) follow from (\ref{(dotsum)}). 
\end{proof}
In the following we always assume that the hypothesis of Theorem \ref{Wifund1} hold. 
In order to take advantage of equivalence (\ref{(Iformachiusa1)}), we denote
with $A$ the matrix of umbral monomials such that
\begin{equation}
\Tr(A) = -1 \, \punt \, \beta \, \punt \, \left( \dot{+}_{k=1}^p \nu_k \theta_k \bar{u}_k \right).
\label{(noncentralterm)}
\end{equation}
\begin{rem}{\it  Central Wishart distribution.} \label{rem1}
When $M=0$ in (\ref{(wishnoncentral)}), then $W(n) = \widehat{W}(n) =  X_{(1)}^\dag X_{(1)} + \cdots + X_{(n)}^\dag X_{(n)}$
with $\{X_{(1)}^\dag X_{(1)}, \ldots, X_{(n)}^\dag X_{(n)}\}$ i.i.d. random matrices of order $p.$
We write $\widehat{W}(1)$ to denote one of the random matrices $X_{(k)}^\dag X_{(k)}.$ From Theorem \ref{Wifund1},
if $b_{kk}=0$ for $k=1,2,\ldots,p$ then $\nu_1 \equiv \cdots \equiv \nu_p \equiv \varepsilon,$ and
\begin{equation}
\Tr[\widehat{W}(n)] \equiv n \, \punt \, ({+}_{k=1}^p \theta_k \bar{u}_k) \equiv n \, \punt \, (\theta_1 \bar{u}_1) +
\cdots + n \, \punt \, (\theta_p \bar{u}_p).
\label{(rapcentralwish)}
\end{equation}
Then equivalence (\ref{(Iformachiusa1)}) separates the contribution of the non-centrality matrix $\Omega$
included in $A$ from the central Wishart distribution $\widehat{W}(n).$ This remark was already done in \cite{Anderson} without giving
an explicit expression to the convolution. This because the symbol in (\ref{(Iformachiusa1)}) does not have a probabilistic
counterpart. Indeed if it is true that  $n \, \punt \, ({+}_{k=1}^p \theta_k \bar{u}_k)$ corresponds to a central Wishart distribution, $\Tr(A)$ in (\ref{(noncentralterm)})
is just a formal compound Poisson r.v. \cite{Dinsen}, since its parameter $-1$ is negative.
\end{rem}
Proposition \ref{Wifund12} states the connection between central Wishart distributions (\ref{(rapcentralwish)}) and partition umbrae.
\begin{prop} \label{Wifund12}
$\Tr[\widehat{W}(n)] \equiv  n \, \punt \, \beta \, \punt \, \left( \dot{+}_{k=1}^p \varkappa_{\scriptsize{\theta_k \bar{u}_k}} \right).$
\end{prop}
\begin{proof}
Since $\beta \, \punt \, \chi \equiv u$ \cite{Bernoulli}, then $n \, \punt \, \left( {+}_{k=1}^p \theta_k \bar{u}_k \right) \equiv n \, \punt  \, \beta \, \punt  \, \chi \, \punt \, \left( {+}_{k=1}^p \theta_k \bar{u}_k \right).$ 
The result follows from Theorem \ref{Wifund1} by recalling the additivity property of cumulants
$$\chi \, \punt \, \left( {+}_{k=1}^p \theta_k \bar{u}_k \right) \equiv \dot{+}_{i=1}^p \chi \, \punt \, (\theta_k \bar{u}_k ) \equiv \dot{+}_{i=1}^p \varkappa_{\theta_k \bar{u}_k}.$$
\end{proof}
\begin{thm}[Sheffer polynomial sequence\footnote{In terms of g.f.'s, Sheffer polynomial sequences $\{s_i(x)\}$ are such that $\sum_{i \geq 1} s_i(x) \frac{z^i}{i!} =g(z) \exp [x f(z)]$ with $g(z)$ and $f(z)$ formal power series such that $g(0)=1$ and $f(0)=0.$ For further references on Sheffer polynomial sequences see \cite{Niederhausen}.}] \label{1c}
The sequence $\{E[\Tr(W(n))^k]\}$ is of Sheffer polynomial type with $x$ replaced by the nonnegative integer $n.$
\end{thm}
\begin{proof}
The umbra $\alpha + x \, \punt \, \beta \, \punt \, \gamma^{\cop}$ is a Sheffer umbra \cite{Nied}. The result follows from Theorem \ref{Wifund1} and Proposition \ref{Wifund12}, choosing as umbra $\alpha$ the umbra  $-1 \punt \, \beta \punt \left( \dot{+}_{k=1}^p \nu_k \theta_k \bar{u}_k \right)$ and as umbra $\gamma^{\cop}$ the umbra  $\dot{+}_{k=1}^p \varkappa_{\scriptsize{\theta_k \bar{u}_k}},$ given in Proposition \ref{Wifund12}. Observe that this umbra has compositional inverse since $E[\dot{+}_{k=1}^p \varkappa_{\scriptsize{\theta_k \bar{u}_k}}]=\Tr(\Sigma) \ne 0.$
\end{proof}
If $b_{ii}=0$ for $i=1,2,\ldots,p,$ then $-1 \punt \, \beta \, \punt \left( \dot{+}_{k=1}^p \nu_k \theta_k \bar{u}_k \right) \equiv \varepsilon$ and $\{E[\Tr(W(n))^k]\}$ is a binomial sequence:
$E( \Tr[\widehat{W}(n+m)]^i) = \sum_{k=0}^i \binom{i}{k} E (\Tr[\widehat{W}(n)]^i)$ $E (\Tr[\widehat{W}(m)]^{i-k}).$
\begin{cor} [Binomial polynomial sequence] \label{2c}
The sequence $\{E(\Tr[\widehat{W}(n)]^k) \}$ is the binomial sequence associated to $\{E(\Tr[W(n)]^k)\}.$
\end{cor}
Since the m.g.f. of $\Tr[\widehat{W}(n)]$ is convergent in a suitable neighborhood of $0,$ the distribution of $\Tr[\widehat{W}(n)]$ is univocally determined by  its moments.
Thanks to the latter equivalence in (\ref{(distrib)}), with the umbra $\alpha$ replaced by $+_{k=1}^p \theta_k \bar{u}_k,$
a further consequence of Corollary \ref{2c} is the following result.
\begin{cor} \label{0c}
$\Tr[\widehat{W}(n+m)] \, \stackrel{d}{=} \, \Tr[\widehat{W}(n)] + \Tr[\widehat{W}(m)].$
\end{cor}
Proposition \ref{(sim)} gives a Sheffer identity (cf. \cite{Niederhausen}) for $\Tr[{W}(n)].$
\begin{prop} \label{(sim)}
$\Tr[W(n+m, \Sigma, M)] \, \stackrel{d}{=} \, \Tr[W(n, \Sigma, M)] + \Tr[\widehat{W}(m)].$
\end{prop}
\begin{proof}
The result follows from equivalence (\ref{(rapcentralwish)}) by observing that
\begin{eqnarray*}
\Tr[W(n+m)] & \equiv & -1 \punt \, \beta \, \punt \left( \dot{+}_{k=1}^p \nu_k \theta_k \bar{u}_k \right) + (n+m) \punt  \left(+_{k=1}^p  \theta_k \bar{u}_k
\right)\\
& \equiv&  \left[ -1 \punt \, \beta \, \punt  \left( \dot{+}_{k=1}^p \nu_k \theta_k \bar{u}_k \right) + n \punt \left(+_{k=1}^p  \theta_k \bar{u}_k
\right) \right] + m \punt \left(+_{k=1}^p  \theta_k \bar{u}_k \right).
\end{eqnarray*}
\end{proof}
\begin{prop} \label{(sim1bis)}
If $M = M_1 + M_2$ with $M_1=\sum_{i=1}^n \boldsymbol{m}_i^\dag \boldsymbol{m}_i$ and $M_2 = \sum_{i=n+1}^{n+m} \boldsymbol{m}_i^\dag
\boldsymbol{m}_i,$ then $\Tr[W(n+m, \Sigma, M)] \stackrel{d}{=} \Tr[W(n,\Sigma, M_1)] + \Tr[W(m,\Sigma, M_2)].$
\end{prop}
\begin{proof}
Assume $\Tr(M_1) = b_{11} + \cdots + b_{pp}$ and $\Tr(M_2) = \tilde{b}_{11} + \cdots + \tilde{b}_{pp}$ with $\Tr(M_1) + \Tr(M_2) = \Tr(M).$  The result follows from
Theorem \ref{Wifund1} by observing that the umbral monomials $\nu_k$ in (\ref{(umbni)}) can be splitted in the summation of the umbral monomials $\tilde{\nu}_k \equiv
(\chi \punt \, b_{kk} \punt \, \beta)$ and $\bar{\nu}_k \equiv (\chi \,\punt \tilde{b}_{kk} \, \punt \beta),$ since $\chi \punt \, (b_{kk}+\tilde{b}_{kk}) \punt \, \beta \equiv
(\chi \punt \, b_{kk}+ \chi \punt \, \tilde{b}_{kk}) \punt \, \beta \equiv \chi \punt \, b_{kk} \punt \, \beta + \chi \punt \, \tilde{b}_{kk} \punt \, \beta.$
This last equivalence is a consequence of (\ref{(distrib)}) which holds also when $n$ and $m$ are replaced by some umbrae $\delta$ and $\gamma,$ cf. \cite{Dinardoeurop}.
 \end{proof}
One feature of the umbral calculus is the chance to replace the nonnegative integer $n$ in the dot-product $n \punt \, \alpha$
with an umbra. Then the umbral counterpart of a randomized non-central Wishart distribution $W(N) = W_p(N, \Sigma, M)$
can be recovered from Theorem 1, with
\begin{equation}
W(N) = \sum_{i=1}^N (X_{(i)} - \boldsymbol{m})^\dag (X_{(i)} - \boldsymbol{m}),
\label{(rnwd)}
\end{equation}
$\boldsymbol{m}$ a vector of complex numbers and $M = \boldsymbol{m}^\dag  \boldsymbol{m}.$

\begin{prop} \label{(sim1)}
If $N$ is a r.v. with moment sequence umbrally represented by an umbra $\alpha,$ then $\{E(\Tr[W(N)]^k)\}$ is
umbrally represented by
$$\Tr[W( \alpha )] \equiv \alpha \, \punt \, \left[  -1 \punt \, \beta \, \punt \, \left( \dot{+}_{k=1}^p \nu_k \theta_k \bar{u}_k  \right) + \left( {+}_{i=1}^p \theta_i \bar{u}_i \right) \right].$$
\end{prop}
\begin{proof}
The m.g.f. of a random sum $S_N = X_1 + X_2 + \cdots + X_N$ is the composition of the m.g.f. $g(z)$ corresponding to the index $N$ and of the m.g.f. $h(z)$ corresponding to $\{X_i\},$ weighted by the $\log$ function.  The umbral counterpart of $S_N$ is the dot product $\alpha \, \punt \, \gamma,$ where $\alpha$ represents the sequence of moments of $N$ and $\gamma$ represents the sequence of moments of $\{X_i\}.$ From (\ref{(rnwd)}) $\Tr[W(N)] = \sum_{i=1}^N \Tr[(X_{(i)} - \boldsymbol{m})^\dag (X_{(i)} - \boldsymbol{m})]$, so that the sequence $\{E(\Tr[W(N)^k])\}$ is represented by the dot product $\alpha \, \punt \, \gamma,$ with $\alpha$ the umbral counterpart of $N$ and $\gamma$ the umbral counterpart of $\Tr[(X_{(i)} - \boldsymbol{m})^\dag (X_{(i)} - \boldsymbol{m})].$ The result follows from (\ref{(Iformachiusa1)}) and Remark \ref{rem1} since $\gamma \equiv  {+}_{i=1}^p \theta_i \bar{u}_i.$
\end{proof}
\subsection{Moments and cumulants of Wishart distributions}
\subsubsection{Central distributions.} Moments of central Wishart distributions can be expressed by means of
complete homogeneous symmetric polynomials $\{h_i\}$ given in (\ref{(omo)}). Indeed,  from (\ref{(omo)})
and equivalence (\ref{(rapcentralwish)}) we have $E(\Tr[\widehat{W}(1)]^i) = i! \, h_i(\theta_1, \ldots, \theta_p)$ and the generalization to $\Tr[\widehat{W}(n)]$ follows by using the
dot-product with $n.$ From (\ref{(sum3)}) we have
\begin{equation}
E(\Tr[\widehat{W}(n)]^i) = \sum_{\lambda \vdash i} (n)_{l(\lambda)} \, \tilde{d}_{\lambda} \, h_{\lambda}(\theta_1, \ldots, \theta_p) \,\,
\hbox{with  $h_{\lambda}= h_1^{r_1} h_2^{r_2} \cdots$ and $\tilde{d}_{\lambda} = \displaystyle{\frac{i!}{r_1! r_2! \cdots}}.$}
\label{(momwishcent)}
\end{equation}
\begin{rem}\label{cyclic} {\sl Cyclic polynomials.}  Thanks to (\ref{(momwishcent)}), moments of $\Tr[\widehat{W}(n)]$ can be expressed in terms of cyclic polynomials
\begin{equation}
{\mathcal C}_i(x_1, \ldots, x_i) = \sum_{\lambda \vdash i} {\mathfrak c}_{\lambda} x_1^{r_1} \cdots x_i^{r_i} \quad
\hbox{with} \quad {\mathfrak c}_{\lambda}= \frac{i!}{1^{r_1} \, r_1! \, 2^{r_2} \, r_2! \cdots}.
\label{(cyclicpolynomials)}
\end{equation}
A well-known relation between cyclic polynomials and complete homogeneous polynomials is ${\mathcal C}_i(s_1, \ldots,s_i) = i! \, h_i(x_1, \ldots,x_p),$ with $\{s_j\}$  power sum symmetric polynomials in the indeterminates $x_1, \ldots,x_p,$ cf. \cite{MacDonald}. By replacing the indeterminates $\{x_k\}$ with the eigenvalues of $\Sigma,$ we have ${\mathcal C}_i[\Tr(\Sigma), \ldots, \Tr(\Sigma^i)] = i! \, h_i(\theta_1,  \ldots, \theta_p) = E(\Tr[\widehat{W}(1)]^i)$ and from (\ref{(momwishcent)})
\begin{equation}
E(\Tr[\widehat{W}(n)]^i) = \sum_{\lambda \vdash i} (n)_{l(\lambda)} \, d_{\lambda} \, {\mathcal C}_{\lambda}[\Tr(\Sigma), \ldots, \Tr(\Sigma^i)]
\quad \hbox{ with ${\mathcal C}_{\lambda}= {\mathcal C}_1^{r_1} {\mathcal C}_2^{r_2} \cdots.$}
\label{(momwishcent1)}
\end{equation}
\end{rem}
\begin{prop} \label{Letac}
If $\{ \Tr(\Sigma^i)\}$ is umbrally represented by $\sigma,$ then
$\Tr[\widehat{W}(n)] \equiv n \, \punt \, \beta \, \punt \, (\sigma \varkappa_{\bar{u}}).$
\end{prop}
\begin{proof}
From (\ref{comp}), with $\gamma$ replaced by the unity umbra $u,$ we have $E[(\beta \, \punt \, (\alpha \varkappa_{\bar{u}}))^i] = \sum_{\lambda \vdash i} d_{\lambda} \, a_{\lambda} \,
{\mathfrak u}_{\lambda}$ with ${\mathfrak u}_{\lambda} = (2-1)!^{r_2} (3-1)!^{r_3} \cdots.$ By observing that ${\mathfrak c}_{\lambda} = d_{\lambda} {\mathfrak u}_{\lambda}$
in (\ref{(cyclicpolynomials)}),  we have
\begin{equation}
E[(\beta \, \punt \, (\alpha \varkappa_{\bar{u}}))^i] = {\mathcal C}_i(a_1, \ldots, a_i).
\label{(cyclicpartition)}
\end{equation}
Then $\Tr[\widehat{W}(1)] \equiv \beta \, \punt \, (\sigma \varkappa_{\bar{u}})$ and
the result follows since $f(\Tr[\widehat{W}(n)],z) = f(\Tr[\widehat{W}(1)],z)^n.$
\end{proof}
As a consequence, the symbolic representation of a randomized central Wishart distribution is
$\alpha \, \punt \, \beta \, \punt \, (\sigma \varkappa_{\bar{u}})$ with moments given in (\ref{comp}).
\begin{rem} \label{cyclicBell} {\rm A different expression for moments of central Wishart distributions is given in \cite{Withers} where vectors and matrix exponential Bell polynomials are employed and introduced as derivatives of a function of a vector or matrix function of a vector or matrix. On the other hand there is a well-known relation between cyclic polynomials and complete Bell exponential polynomials \cite{MacDonald}, that is ${\mathcal C}_i(a_1, a_2, \ldots, a_i) = Y_i [a_1, a_2 1!, a_3 2!, \ldots, a_i (i-1)! ].$ Due to (\ref{(cyclicpartition)}), the partition umbra is the bridge between the two families of polynomials, since $E[(\beta \, \punt \, \alpha)^i]=Y_i(a_1, \ldots, a_i),$ cf. \cite{Dinsen}.}
\end{rem}
Proposition \ref{Letac} is in agreement with equation (2.4) of \cite{Letac1} where permutations\footnote{A permutation $\sigma \in {\mathfrak S}_i$ of $[i]$, with ${\mathfrak S}_i$ the symmetric group, can be decomposed into disjoint cycles. The length of the cycle $c \in C(\sigma)$ is its cardinality, denoted by ${\mathfrak l}(c)$. The number of cycles is denoted by $|C(\sigma)|$. A permutation $\sigma$ with $r_1$ $1$-cycles, $r_2$ $2$-cycles and so on is said to be of cycle class $\lambda=(1^{r_1},2^{r_2},\ldots) \vdash i.$ The integer ${\mathfrak c}_{\lambda}$
in (\ref{(cyclicpolynomials)}) denotes the number of permutations $\sigma \in {\mathfrak S}_i$ of cycle class $\lambda=(1^{r_1},2^{r_2},\ldots) \vdash i.$ When misunderstandings occur, we refer to the cycle class of a permutation $\sigma$ as $\lambda_{\sigma}.$ In particular we have $l({\lambda_{\sigma}})=|C(\sigma)|.$} are employed in place of integer partitions. Indeed indexing equation (\ref{(momwishcent1)}) in permutations,
$$E[(n \, \punt \, \beta \, \punt \, (\sigma \varkappa_{\bar{u}}))^i] = \sum_{\sigma \in {\mathfrak S}_i} n^{{\mathfrak l}(\sigma)}
\prod_{c \in C(\sigma)} {\mathcal C}_{{\mathfrak l}(c)} (s_1, \ldots, s_{{\mathfrak l}(c)})$$ and equation (2.4) of \cite{Letac1} follows for $h_1 = \cdots = h_k = I_p$ since ${\mathcal C}_i(s_1, \ldots, s_i) = \sum_{\sigma \in {\mathfrak S}_i} \prod_{c \in C(\sigma)} s_{{\mathfrak l}(c)}.$

As often happens with r.v.'s, the sequence of cumulants of central Wishart distributions is more manageable of its moments. In the following, $\hbox{\rm Cum}_i(\Tr[\widehat{W}(n)])$ denotes its $i$-th cumulant for all
positive integer $i,$ and  $\varkappa_{\scriptsize{\Tr[\widehat{W}(n)]}}$ denotes the $\Tr[\widehat{W}(n)]$-cumulant umbra, that is the umbra whose moments are $\hbox{\rm Cum}_i(\Tr[\widehat{W}(n)]).$
\begin{prop} \label{cumwish} $\hbox{\rm Cum}_i(\Tr[\widehat{W}(n)]) = n (i-1)! \Tr(\Sigma^i).$
\end{prop}
\begin{proof}
From Proposition \ref{Letac}, $\hbox{\rm Cum}_i(\Tr[\widehat{W}(n)]) = E[(\chi \, \punt \, n \, \punt \, \beta \, \punt \, (\sigma \varkappa_{\bar{u}}))^i].$ The result follows from (\ref{comp}), with $\gamma$ replaced by the umbra $\chi \, \punt \, n.$ Indeed $E[(\chi \, \punt \, n)^{l(\lambda)}] \ne 0$ only with the partition $\lambda$ of $i$ having length $1.$ In this case, $E[(\chi \, \punt \, n)^{l(\lambda)}]=n$ and only the $i$-th moment $E[(\sigma \varkappa_{\bar{u}})^i] = (i-1)! \Tr(\Sigma^i)$ gives contribution in (\ref{comp}).
\end{proof}
Moments from cumulants can be recovered by using the partition umbra since if $\varkappa \equiv \chi \, \punt \, \alpha$ then $\alpha \equiv \beta \, \punt \, \varkappa,$ cf. \cite{Bernoulli}.
\begin{cor} \label{(aabb)}
$\Tr[\widehat{W}(n)] \equiv \beta \, \punt \, \varkappa_{\scriptsize{\Tr[\widehat{W}(n)]}} \equiv n \, \punt \, \beta \, \punt \, \varkappa_{\scriptsize{\Tr[\widehat{W}(1)]}}.$
\end{cor}
\subsection{Non-central distributions.}
From Theorem \ref{Wifund1} and Remark \ref{cyclic}, moments of $\Tr[W(n)]$ can be computed by binomial expansion:
\begin{eqnarray}
E\left(\Tr[W(n)]^i \right) & = & \sum_{j=0}^i \binom{i}{j} E\left[\left\{ -1 \, \punt \, \beta \, \punt \, \left( \dot{+}_{k=1}^p \nu_k  \theta_k  \bar{u}_k \right)\right\}^j \right] E\left[\left\{ n \, \punt \, (+_{k=1}^p  \theta_k \bar{u}_k) \right\}^{i-j} \right] \, \nonumber \\
& = & i! \sum_{j=0}^i \left( \sum_{\lambda \vdash j} \frac{(-1)^{l(\lambda)}}{r_1! r_2! \cdots} \prod_{i=1}^{l(\lambda)} \Tr(\Omega \Sigma^i)^{r_i} \right)
\left( \sum_{\lambda^{\prime} \vdash i-j} \frac{n^{l(\lambda^{\prime})}}{r_1! r_2! \cdots} \prod_{i=1}^{l(\lambda^{\prime})} {\mathcal C}_i^{r_i}(\Sigma) \right)
\label{(mom2)}
\end{eqnarray}
with ${\mathcal C}_i (\Sigma) = {\mathcal C}_i[\Tr(\Sigma), \ldots, \Tr(\Sigma^i)]$ the $i$-th cyclic polynomial. 

From the additivity property of cumulants and Theorem \ref{Wifund1}, cumulants of non-central Wishart distributions are such that $\hbox{\rm Cum}_i(\Tr[W(n)]) = 
\hbox{\rm Cum}_i(\Tr[\widehat{W}(n)]) + \hbox{\rm Cum}_i(\Tr(A)).$ This is the key to prove the following theorem.
\begin{thm} \label{cum2}
$\hbox{\rm Cum}_i(\Tr[W(n)]) = n \, (i-1)! \Tr(\Sigma^i) - i ! \Tr(M \Sigma^{i-1}).$
\end{thm}
\begin{proof}
From Theorem \ref{Wifund1}, cumulants of $\Tr[W(n)]$ are umbrally represented by
$$\chi \, \punt \, [-1 \, \punt \, \beta \, \punt \, (\nu_1 \theta_1 \bar{u}_1 \dot{+} \cdots \dot{+}  \nu_p \theta_p \bar{u}_p) + n \, \punt \, (\theta_1 \bar{u}_1 + \cdots \theta_p \bar{u}_p)].$$
From the additivity property of cumulants this umbra is similar to
$$\chi \, \punt \, [-1 \, \punt \, \beta \, \punt \, (\nu_1 \theta_1 \bar{u}_1 \dot{+} \cdots \dot{+}  \nu_p \theta_p \bar{u}_p)]
\dot{+} \chi \, \punt \, n \, \punt \, (\theta_1 \bar{u}_1 + \cdots + \theta_p \bar{u}_p).$$ 
Since $\chi \, \punt \, n \, \punt \, (\theta_1 \bar{u}_1 + \cdots + \theta_p \bar{u}_p) \equiv \chi \, \punt \, n \, \punt \, \beta \, \punt \, (\sigma \varkappa_{\bar{u}}),$
its moments are given in Proposition \ref{cumwish}. Moreover $E\left\{(\chi \, \punt \, [-1 \, \punt \, \beta \, \punt \, (\nu_1 \theta_1 \bar{u}_1 \dot{+} \cdots \dot{+}  \nu_p \theta_p \bar{u}_p)])^i \right\} = \sum_{\lambda \vdash i} E[(-\chi)^{l(\lambda)}] \, d_{\lambda} \, \lambda! \, {\mathcal S}_{\lambda} = - i! {\mathcal S}_{i},$ with $\{\mathcal{S}_i\}$ given in (\ref{(ccc11)}) and $\lambda!=\lambda_1! \, \lambda_2! \, \cdots$ since $E[(-\chi)^{l(\lambda)}] \ne 0$ only for $l(\lambda)=1.$
\end{proof}
\begin{cor}\label{corex}
$\hbox{\rm Cum}_i(\Tr[W(n)]) = (i-1)! \sum_{j=1}^p [n - i \, b_{j \, j}] \, \theta_j^i.$
\end{cor}
\begin{example}
Set $n=3.$ By using R, if 
$$\Sigma = \left( \begin{array}{ccc}
0.025 & - 0.0075 i & 0.00175 \\
0.0075 i & 0.0070 & 0.00135 \\
0.00175 & 0.00135 & 0.00043
\end{array} \right) \, \hbox{and} \, M = \left( \begin{array}{ccc}
 0.0001  & 0.0210 & 0.3000 \\
 0.0400  & 0.0005 & 0.0200 \\
 0.0010  & 0.0100 & 0.0004
\end{array} \right)$$ then
$$\Omega = \Sigma^{-1} \, M = \left( \begin{array}{rrr}
    -5.16 -  7.45 \, i &   6.04  + 6.25  \, i & -25.10 -   3.79 \, i \\
   -17.19 -  0.80 \, i &  17.82  - 3.98  \, i & -53.27 -  60.37 \, i \\
    77.33 + 32.87 \, i & -57.30 - 12.92  \, i & 269.96 - 174.11 \, i 
\end{array} \right)$$
$$B = \left( \begin{array}{rrr}
      0.63 -   0.05  \,i &    0.99 -   3.68 \, i &   7.22 - 6.01 \, i \\
      0.42 +   8.29  \,i &    3.63 -   4.31 \, i &  16.81 - 9.90 \, i \\
     73.21 -  65.72  \,i &   19.18 -  25.95 \, i & 278.35 - 181.19 \, i
\end{array} \right)$$ and from Corollary \ref{corex} we have $\hbox{\rm Cum}_1(\Tr[W(3)]) = 3 \Tr(\Sigma) - \Tr(M) = 0.03143, \,\,\, \hbox{\rm Cum}_2(\Tr[W(3)]) = 3\, \Tr(\Sigma^2) - 2 \Tr(M \, \Sigma) = 0.0012 + 0.0028 \, i, \hbox{\rm Cum}_3(\Tr[W(3)]) = 
6 \Tr( \Sigma^3) - 6 \Tr(M \Sigma^2) = 1.0192\times 10^{-4} + 3.3278 \times 10^{-6} \, i$ and so on.
\end{example}
Complete Bell exponential polynomials give moments in terms of cumulants \cite{Bernoulli}: if $\{c_i\}$
is the sequence of formal cumulants of $\{a_n\},$ then $a_n=Y_n(c_1,  \ldots,c_n).$ The connection with cyclic polynomials is a consequence of
Remark \ref{cyclicBell}.
\begin{cor} \label{noncentral} If $\tilde{c}_i = n \, (i-1)! \, \Tr(\Sigma^i)$ and $\bar{c}_i = - i! \, \mathcal{S}_i,$ then
$$E(\Tr[W(n)]^i) = Y_i(\tilde{c}_1 + \bar{c}_1, \ldots, \tilde{c}_i + \bar{c}_i) = {\mathcal C}_i \left(\tilde{c}_1 + \bar{c}_1,  \ldots, \displaystyle{\frac{\tilde{c}_i + \bar{c}_i}{(i-1)!}} \right).$$
\end{cor}
A completely different expression of $E(\Tr[W(n)]^i)$ is given in \cite{Numata}, where weighted generating functions of special graphs are employed.

In \cite{DMS}, a different family of cumulants has been introduced, representing $\{\hbox{\rm Cum}_i(\Tr[W(n)])\}$ normalized to the dimension $p$ of $W(n).$ Indeed the definition given in  
\cite{DMS} leads to the introduction of umbral polynomials $\{\mathfrak{c}_{1,\boldsymbol{\mu}}, \ldots, \mathfrak{c}_{p,\boldsymbol{\mu}} \}$ with $\boldsymbol{\mu} = (\mu_1, \ldots, \mu_p)$ and $\Tr[W(n)] \equiv \mu_1 + \cdots + \mu_p,$ such that
\begin{equation}
\Tr[W(n)] \equiv p \, \punt \, \beta \, \punt \, \Tr[\, {\mathfrak C}(n)]
\label{(clascum)}
\end{equation}
where $\Tr[\, {\mathfrak C}(n)]=\mathfrak{c}_{1,\boldsymbol{\mu}} + \cdots + \mathfrak{c}_{p,\boldsymbol{\mu}}.$
Note that the order $p$ plays a fundamental role in (\ref{(clascum)}), differently from Theorem \ref{cum2} where instead the integer $n$ gives the main contribution.
The following theorem highlights the connection between the two families of cumulants.
\begin{thm} \label{relcum} $\varkappa_{\scriptsize{\Tr[W(n)]}} \equiv \chi \, \punt \, p \, \punt \, \beta \, \punt \Tr[\,{\mathfrak C}(n)]$ and $\Tr[\,{\mathfrak C}(n)] \equiv \chi \, \punt \, \frac{1}{p} \, \punt \, \beta 
\, \punt \, \varkappa_{\scriptsize{\Tr[W(n)]}}.$
\end{thm}
\begin{proof}
The first equivalence follows from (\ref{(clascum)}) by taking the dot-product with $\chi$ of both sides.
The latter follows from the first by taking the dot-product of both sides with $\chi \punt \frac{1}{p}.$
\end{proof}
\section{Joint moments and joint cumulants of Wishart distributions}
In order to deal with joint moments (\ref{(joint)}), we need to recall the multivariate version of the symbolic method.
As before, we just introduce the tools necessary to work with the object of this paper. The reader interested in a detailed exposition of the topic is referred to \cite{faa}.
\subsection{Multivariate symbolic method.}
Let $\{\nu_1, \ldots, \nu_m\}$ be a set of umbral monomials
with supports not necessarily disjoint.  A sequence $\{g_{\ibs}\}_{\ibs \in \mathbb{N}_0^m} \in {\mathbb C}$, with $g_{\ibs} = g_{i_1, i_2, \ldots, i_m}$ and $g_{\bf 0} = 1$, is
represented by the $m$-tuple $\nubs=(\nu_1,\ldots,\nu_m)$ if
\begin{equation}
E[\nubs^{\ibs}] = g_{\ibs},
\label{(multmoments)}
\end{equation}
for each multi-index $\ibs \in \mathbb{N}_0^m$. The elements $\{g_{\ibs}\}_{\ibs \in \mathbb{N}_0^m}$ in (\ref{(multmoments)}) are called {\it multivariate moments} of $\nubs$. Paralleling (\ref{(gf1)}), the g.f. of $\nubs$ is the exponential formal power series
\begin{equation}
e^{\nu_1 z_1 +  \cdots + \nu_m z_m} = u + \sum_{k \geq 1} \, \sum_{{|\ibs|=k}} \nubs^{\ibs}\frac{\zbs^{\ibs}}{\ibs!}
\in {\mathbb C}[\A][[z_1, \ldots,z_m]]
\label{(gf)}
\end{equation}
with $\zbs = (z_1,  \ldots, z_m), |\ibs|=i_1 + \cdots + i_m$ and $\ibs!=i_1! \cdots i_m!$.
If the sequence $\{g_{\ibs}\}$ is umbrally represented by $\nubs$ and has (exponential) g.f.
\begin{equation}
f(\zbs) = 1 + \sum_{k \geq 1} \,  \sum_{|\ibs|=k} g_{\ibs} \frac{\zbs^{\ibs}}{\ibs!}
\label{(genfun1)}
\end{equation}
then from (\ref{(gf)}) $E[e^{\nu_1 z_1 + \cdots + \nu_m z_m}] = f(\zbs).$ Taking into account (\ref{(multmoments)}), the g.f. in (\ref{(genfun1)}) is denoted by $f(\nubs,\zbs)$.
Two umbral vectors $\nubs_1$ and $\nubs_2$ are said to be {\it similar}, in symbols $\nubs_1 \equiv \nubs_2$, iff
$f(\nubs_1,\zbs)=f(\nubs_2,\zbs)$, that is $E[\nubs_1^{\ibs}]=E[\nubs_2^{\ibs}]$ for all $\ibs \in \mathbb{N}_0^m$.
They are said to be {\it uncorrelated} iff $E[\nubs_1^{\ibs} \nubs_2^{\jbs}]= E[\nubs_1^{\ibs}]E[\nubs_2^{\jbs}]$ for all $\ibs, \jbs \in  \mathbb{N}_0^m$.

An analogous of (\ref{(sum3)}) holds for the multivariate case, provided to replace integer partitions with multi-index partitions\footnote{
A partition of a multi-index $\ibs,$  in symbols $\lambdabs \mmodels \ibs,$ is a matrix $\lambdabs = (\lambda_{ij})$ of nonnegative integers and
with no zero columns in lexicographic order  such that
$\lambda_{r1}+\lambda_{r2}+\cdots+\lambda_{rk}=i_r$ for
$r=1,2,\ldots,n$. The number of columns of $\lambdabs$ is called the length of $\lambdabs$
and denoted by $l(\lambdabs)$.  As for integer partitions, the notation $\lambdabs = (\lambdabs_{1}^{r_1}, \lambdabs_{2}^{r_2}, \ldots)$
means that in the matrix $\lambdabs$ there are $r_1$ columns equal to $\lambdabs_{1}$,
$r_2$ columns equal to $\lambdabs_{2}$ and so on, with $\lambdabs_{1} <
\lambdabs_{2} < \cdots$. We call $r_i$ multiplicity of $\lambdabs_i$ and set
$\ml(\lambdabs)=(r_1, r_2,\ldots)$.} \cite{faa}. The dot-product $n \, \punt \, \nubs$ of a nonnegative integer $n$ and a $m$-tuple $\nubs$ denotes
the summation $\nubs^{\prime} + \nubs^{\prime \prime} + \cdots + \nubs^{\prime \prime \prime}$  with
$\{\nubs^{\prime}, \nubs^{\prime \prime}, \ldots, \nubs^{\prime \prime \prime}\}$ a set of $n$ uncorrelated and similar $m$-tuples.
For $\ibs \in {\mathbb N}^m_0,$ its multivariate moments are
\begin{equation}
E[(n \, \punt \, \nubs)^{\ibs}] =  \sum_{\lambdabs \mmodels \ibs}
d_{\lambdabs} \, (n)_{l(\lambdabs)} \, g_{\lambdabs}, \qquad \hbox{with} \qquad d_{\lambdabs} = \frac{\ibs!}{\ml(\lambdabs)! \lambdabs!}, \label{(eq:15)}
\end{equation}
the sum runs over all partitions  $\lambdabs = (\lambdabs_{1}^{r_1}, \lambdabs_{2}^{r_2}, \ldots)$ of the multi-index $\ibs, g_{\lambdabs} = g_{\lambdabs_1}^{r_1} g_{\lambdabs_2}^{r_2} \cdots$
and $g_{\lambdabs_i} = E[\nubs^{\lambdabs_i}]$. If we replace the integer $n$ in (\ref{(eq:15)}) with the dot-product
$\alpha \, \punt \, \beta,$ we get the auxiliary umbra $\alpha \, \punt \, \beta \, \punt \, \nubs$
whose moments are
\begin{equation}
E[(\alpha \, \punt \, \beta \, \punt \, \nubs)^{\ibs}] =  \sum_{\lambdabs \mmodels \ibs}
d_{\lambdabs}  \, a_{l(\lambdabs)} \, g_{\lambdabs} , \label{(eq:16)}
\end{equation}
with $\{a_i\}$ in $a_{l(\lambdabs)}$ umbrally represented by $\alpha$.
In particular the g.f. of the auxiliary umbra $\alpha \, \punt \, \beta \, \punt \, \nubs$ turns to be
$f(\alpha \, \punt \, \beta \, \punt \, \nubs,\zbs)=f[\alpha, f(\nubs,\zbs)-1],$ that is
the composition of the univariate g.f. $f(\alpha,z)$ and the multivariate g.f. $f(\nubs,\zbs).$
More details on the composition of multivariate formal power series are given in \cite{faa}. For what we need in the following, we just recall that
$f( -1 \, \punt \, \beta \, \punt \, \nubs,\zbs)= \exp\{ - [f(\nubs,\zbs)-1] \}$ and $f( n \, \punt \, \beta \, \punt \, \nubs,\zbs)= \exp\{ n [f(\nubs,\zbs)-1] \}.$  If in
(\ref{(eq:16)}) the umbra $\alpha$ is replaced by the umbra $\chi \, \punt \, \chi,$ then the $\nubs$-cumulant umbra $\chi \, \punt \, \chi \, \punt \, \beta \, \punt \, \nubs
\equiv \chi \, \punt \, \nubs$ is recovered, whose moments are the multivariate cumulants of the $m$-tuple $\nubs.$
\subsection{Joint moments and necklaces.}
The aim of this paragraph is to find a symbolic representation of  joint moment (\ref{(joint)})
as the $\ibs$-th coefficient of the g.f. $E \left( \exp \{ \Tr [ W(n)(H_1 z_1 +  \cdots + H_m z_m)] \} \right),$
with $\ibs =(i_1, \ldots, i_m).$ It is well-known  that if $Z$ is a $p \times p$ Hermitian parameter matrix, then
\begin{equation}
E\left( \exp \{ \Tr[W(n) \, Z] \} \right) = \frac{\exp \left\{ - \Tr[(I_p - \Sigma \, Z)^{-1} \,  \Omega \, \Sigma \, Z] \right\} }{[\det{(I_p - \Sigma \, Z)}]^n}
\label{(laptrasfbis)}
\end{equation}
with $\Omega$ the non-centrality matrix \cite{Numata}. Within formal power series, $f(W(n) \, H,Z) = f(W(n),$ $H \, Z)$ with $H \in {\mathbb C}^{p \times p}.$ Then choose
$H$ and $Z$ such that $\Tr[W(n) \, H \, Z] = \Tr[W(n) \, H_1] \, z_1 + \cdots + \Tr[W(n) H_m] \, z_m.$ From (\ref{(laptrasfbis)})
\begin{equation}
E\left[ \exp \left( \Tr[W(n) \, H \, Z] \right) \right] = \frac{\exp\left\{- \Tr[(I_p -  \Sigma \, H \, Z)^{-1} \, \Omega \, \Sigma \,  H\, Z] \right\}}{[\det{(I_p - \Sigma \, H \, Z)}]^n}.
\label{(laptrasfter)}
\end{equation}
Equation (\ref{(laptrasfter)}) is the starting point to prove Theorem \ref{Wifund2}, where the
symbolic representation of a non-central Wishart distribution is generalized to the multivariate case. The form of this
symbolic representation is very similar to the univariate case, provided scalar umbrae are replaced by suitable $m$-tuples of umbral monomials.

In order to take advantage of the cyclic property of traces, we resort the notion of necklace. A $m$-ary necklace is an equivalence class of $m$-ary strings equivalent under rotation (the cyclic
group) \cite{Cattell}. Let the $m$-ary strings be elements of the set $\{1,\ldots,m\}^{|\ibs|},$ then having length $|\ibs|,$ with $\ibs=(i_1, \ldots, i_m).$
Denote the set of all necklaces  of length $|\ibs|$ over the $m$-ary alphabet of indexes $\{1,\ldots,m\}$ with ${\boldsymbol N}_m{(|\ibs}|).$   It is natural to choose as representative ${\mathfrak a}$ of a necklace its
lexicographically smallest string. Then we denote the corresponding necklace with $[{\mathfrak a}]$ .
\begin{defn}
In the set $\{1,\ldots,m\}^{|\ibs|},$ a necklace is said to be of kind $\ibs,$ if the elements in its strings are chosen in the multiset\footnote{Informally, a multiset is a generalization of the notion of set where
elements are allowed to appear more than once. The notation $\{1^{i_1}, 2^{i_2}, \ldots, m^{i_m}\}$ denotes
that $1$ appears $i_1$ times, $2$ appears
$i_2$ times and so on. For a formal definition, cf. \cite{Bernoulli}.} $\{1^{i_1}, 2^{i_2}, \ldots, m^{i_m}\}.$
\end{defn}
To the best of our knowledge, this definition is given here for the first time.
We denote the set of representatives of necklaces of kind $\ibs$ with ${\boldsymbol N}_m{[\ibs]}.$
\begin{example}
If $m=3$ and $|\ibs|=3,$ then ${\boldsymbol N}_3(3)=
\left\{111, 222, 333, 112, 122, 113, 133, 223, 233,\right.$ $\left. 123, 132\right\}.$ The string $123$ is the representative of the necklace $[123]=\{123, 312, 231\}$ and
the string $112$ is the representative of the necklace $[112]=\{112, 121, 211\}.$  Moreover
${\boldsymbol N}_3[(3,0,0)]=\{111\}$ and ${\boldsymbol N}_3[(1,2,0)]=\{122\}.$ We have
$|{\boldsymbol N}_3[\ibs]|=1,$ for all $\ibs$ with $|\ibs|=3,$
except for $\ibs=(1,1,1)$ as ${\boldsymbol N}_3 [(1,1,1)]=\{123,132\}.$
\end{example}
In most cases, a necklace has full size, which is equal to the length of a string. But, there are periodic necklaces of period $d < m,$ a divisor of $m.$
In the example, $111$ is a necklace of period $1.$ The aperiodic necklaces are also known as {\it Lyndon words}. Denote the set of Lyndon words of length $|\ibs|$
over the $m$-ary alphabet of indexes $\{1,\ldots,m\}$ with  ${\boldsymbol M}_m{(|\ibs}|).$ As before, denote the set of aperiodic necklaces of kind $\ibs$
with ${\boldsymbol M}_m [\ibs].$ The closed form formula for joint moments (\ref{(joint)}) is given in the following theorem.

\begin{thm} \label{Wifund2}
For $\ibs = (i_1, \ldots, i_m) \in {\mathbb N}_0^{m}$
\begin{equation}
E\left\{ \Tr \left[ W(n) H_1 \right]^{i_1} \cdots \Tr \left[ W(n) H_m \right]^{i_m} \right\}  = E[(- 1 \, \punt \, \beta \, \punt \tilde{{\boldsymbol \eta}} + n \, \punt \, \beta \, \punt \, \tilde{{\boldsymbol \rho}})^{\ibs}], \label{(jointmult1)}
\end{equation}
with
\begin{description}
\item[{\it i)}] $\tilde{{\boldsymbol \rho}} = (\bar{u}_1 \, \rho_1, \ldots, \bar{u}_m \, \rho_m)$ and ${\boldsymbol \rho}=(\rho_1, \ldots, \rho_m)$ such that
\begin{equation}
E[{\boldsymbol {\rho}}^{\ibs}] = \sum_{{\mathfrak a} \in {\boldsymbol M}_m{[\ibs}]} \Tr\left[ \prod_{k \in {\mathfrak a}} (\Sigma H_k) \right] + \sum_{{\mathfrak a} \in {\boldsymbol N}_m [\ibs] - {\boldsymbol M}_m [\ibs]} \frac{1}{d} \Tr\left[ \prod_{k \in {\mathfrak a}} (\Sigma H_k) \right];
\label{(firstpart)}
\end{equation}
\item[{\it ii)}] $\tilde{{\boldsymbol \eta}} = (\bar{u}_1 \, \eta_1, \ldots, \bar{u}_m \, \eta_m)$ and ${\boldsymbol \eta}=(\eta_1, \ldots, \eta_m)$ such that
\begin{equation}
E[{\boldsymbol \eta}^{\ibs}] =  \sum_{{\mathfrak a} \in {\boldsymbol N}_m [\ibs]} \left( \sum_{b \in [{\mathfrak a}]} \Tr\left[ \Omega \prod_{k \in b} (\Sigma H_k) \right] \right).
\label{(secondpart)}
\end{equation}
\end{description}
\end{thm}
If $|\ibs|$ is prime or ${\boldsymbol N}_m [\ibs] = {\boldsymbol M}_m [\ibs],$ then  (\ref{(firstpart)}) reduces to the first summation.
\begin{proof}
Within formal power series, for a matrix $A$ the following identity holds $[\det(I-A)]^{-1}=\exp \left( \sum_{j \geq 1} \Tr(A^j)/j \right),$ cf. \cite{Stanley1}.
Then taking into account (\ref{(genfun1)}) and (\ref{(laptrasfter)}), with $H Z$ replaced by $H_1 z_1 + \cdots + H_m z_m,$ we have
$\left( \det {[I_p -  \Sigma \, (H_1 z_1 + \cdots + H_m z_m) ]}\right)^{-n} = \exp \left\{ n [f(\tilde{\boldsymbol {\rho}}, {\zbs})-1] \right\}$ with
\begin{equation}
f(\tilde{\boldsymbol {\rho}}, {\zbs}) - 1 =   \sum_{j \geq 1} \frac{1}{j} \sum_{|\ibs|=j} \left[ \sum_{{\mathfrak a} \in
\{1,\ldots,m\}^{|\ibs|}} \Tr \left( \prod_{k \in {\mathfrak a}}  \Sigma H_{k} \right) \right] \, z_1^{i_1} \cdots z_m^{i_m}.
\label{(firstpartII)}
\end{equation}
Equation (\ref{(firstpart)}) follows by indexing the inner summation in (\ref{(firstpartII)}) with the necklaces in ${\boldsymbol N}_m [\ibs]$ for each $\ibs$ and by multiplying the corresponding traces with the cardinality of  necklaces, due to their cyclic property. In (\ref{(laptrasfter)}) since $(I-A)^{-1}= I + \sum_{j \geq 1} A^j,$ then $\Tr [(I_p -  \Sigma \, H \, Z)^{-1} \, \Omega \, \Sigma \,  H \, Z] =  \Tr [\Sigma \,  H \, Z \, (I_p -  \Sigma \, H \, Z)^{-1} \, \Omega] = \sum_{j \geq 1} \Tr[\Omega \, (\Sigma \, H \, Z)^j]$ and for $H Z$ replaced by $H_1 z_1 + \cdots + H_m z_m$ a characterization of $f(\tilde{{\boldsymbol \eta}}, \zbs)$ similar to (\ref{(firstpartII)}) follows
$$
f(\tilde{\boldsymbol {\eta}}, {\zbs}) - 1 =   \sum_{j \geq 1}  \sum_{|\ibs|=j} \left[ \sum_{{\mathfrak a} \in
\{1,\ldots,m\}^{|\ibs|}} \Tr \left( \Omega \prod_{k \in {\mathfrak a}}  \Sigma H_{k} \right) \right] \, z_1^{i_1} \cdots z_m^{i_m}.
$$
Differently from (\ref{(firstpartII)}), we could not group the inner summation with respect to the elements of necklaces, since we could not use the cyclic property of the trace, due to the position of the non-centrality matrix $\Omega.$
\end{proof}
A multinomial expansion is applied to the right-hand-side of (\ref{(jointmult1)}) in order to get
\begin{eqnarray}
E\left\{ \prod_{j=1}^m \Tr \left[ W(n) H_j \right]^{i_j} \right\} & = &  \displaystyle{\sum_{\tbs_1, \tbs_2 \in \mathbb{N}_0^m \atop \tbs_1 + \, \tbs_2 = \ibs}}
\frac{\ibs!}{\tbs_1! \tbs_2!} E[(- 1 \punt \beta \punt \tilde{{\boldsymbol \eta}})^{\tbs_1}] E[(n \punt \beta \punt \tilde{{\boldsymbol \rho}})^{\tbs_2}]
\label{(mommul)} \\
& = & \ibs! \displaystyle{\sum_{\tbs_1, \tbs_2 \in \mathbb{N}_0^m \atop \tbs_1 + \, \tbs_2 = \ibs}}
\left( \sum_{\lambdabs \mmodels \tbs_1}
\frac{(-1)^{l(\lambdabs)}}{{\mathfrak m}(\lambdabs)} \prod_{\lambdabs_i} E[{\boldsymbol \eta}^{\lambdabs_i}]^{r_i} \right)
\left( \sum_{\lambdabs \mmodels \tbs_2}
\frac{(n)^{l(\lambdabs)}}{{\mathfrak m}(\lambdabs)} \prod_{\lambdabs_i} E[{\boldsymbol \rho}^{\lambdabs_i}]^{r_i} \right)
\label{(addendi)}
\end{eqnarray}
where  $E[{\boldsymbol \eta}^{\lambdabs_i}]$ and $E[{\boldsymbol \rho}^{\lambdabs_i}]$ are given
in (\ref{(firstpart)}) and (\ref{(secondpart)}) respectively. The function {\tt MakeTab} in \cite{faa} has been employed for computing all the multi-indexes ${\tbs_1, \tbs_2}$
in (\ref{(addendi)}). A separate {\tt Maple} procedure \cite{Dinardoalg} has been set up in order to expand (\ref{(addendi)}), see the following example.
\begin{example} \label{excum}
For $m=2$ and $\ibs=(1,2),$ an explicit expression of (\ref{(addendi)}) is given in the following:
\begin{eqnarray*}
&  &  E\left\{ \Tr \left[ W(n) H_1 \right] \Tr \left[ W(n) H_2 \right]^2  \right\} =
n \Tr \left( H_{{2}} \right) \Tr \left( \Omega H_{{1}} H_{{2
}} \right) -n \Tr \left( H_{{2}} \right) \Tr \left( \Omega
H_{{2}} H_{{1}} \right) \\
& + &  n \Tr \left( H_{{2}} \right) \Tr
 \left( \Omega H_{{1}} \right) \Tr \left( \Omega H_{{2}} \right)
- n \Tr \left( \Omega H_{{2}} \right) \Tr \left( H_{{1}} H_{{2}} \right)
- n^2 \Tr \left( \Omega H_{{2}} \right) \Tr
 \left( H_{{1}} \right) \Tr \left( H_{{2}} \right)\\
&  + & 2\, \Tr \left( \Omega H_{{2}} \right) \Tr \left( \Omega H_{{1}} H_{{2}}
 \right) + 2 \, \Tr \left( \Omega H_{{2}} \right)\Tr \left(
\Omega H_{{2}} H_{{1}} \right) - \Tr \left( \Omega H_{{1}}
 \right)  \left( \Tr \left( \Omega H_{{2}} \right)  \right) ^{2} -  \Tr \left( \Omega H_{{1}} {H_{{2}}}^{2} \right) \\
 & - & \Tr \left( \Omega H_{{2}}
H_{{1}} H_{{2}} \right) + \Tr \left( \Omega H_{{1}} \right) \Tr \left( \Omega {H_{{2}}}^{2} \right) + 2\,{n}^{2} \Tr \left( H_{{1}} H_{{2}} \right) \Tr \left( H_{{2}} \right) + {n}^{3} \Tr
 \left( H_{{1}} \right)  \left( \Tr  \left( H_{{2}} \right)
 \right) ^{2} \\
& + &  n \Tr  \left( H_{{1}} {H_{{2}}}^{2} \right) +
 n^2/2 \Tr  \left( H_{{1}} \right) \Tr  \left( {H_{{2}}}^{2}
 \right) +{n}^{2} \Tr  \left( \Omega H_{{1}} \right) \left( \Tr
 \left( H_{{2}} \right)  \right) ^{2} - n/2\, \Tr  \left( \Omega H_{{1}} \right)
   \Tr  \left( {H_{{2}}}^{2} \right) \\
 & + &   n \Tr  \left( H_{{1}} \right)  \left( \Tr \left( \Omega H_{{2}}
 \right)  \right) ^{2}- n \Tr \left( H_{{1}} \right)  \Tr
 \left( \Omega {H_{{2}}}^{2} \right).
\end{eqnarray*}
\end{example}
\begin{rem}
{\rm If we replace the nonnegative integer $n$ with a scalar umbra $\alpha$ in (\ref{(jointmult1)}) then
\begin{equation}
E\left\{ \Tr \left[ W( \alpha ) H_1 \right]^{i_1} \cdots \Tr \left[ W(\alpha) H_m \right]^{i_m} \right\}  = E[(- 1 \punt \beta \punt \tilde{{\boldsymbol \eta}} + \alpha \punt \beta \punt \tilde{{\boldsymbol \rho}})^{\ibs}] \label{(jointmult1bis)}
\end{equation}
joint moments of randomized non-central Wishart distributions. Although not explicitly mentioned, this replacement could be performed for all the results given in the following.}
\end{rem}

Joint moments of $\tilde{\boldsymbol \rho}$ correspond to joint moments $E\{ \Tr[ \widehat{W}(n) H_1]^{i_1} \cdots \Tr [\widehat{W}(n) H_m]^{i_m} \}$ of central Wishart matrices  and are a generalization of formulae given in \cite{Letac1}.
Joint moments of $\tilde{\boldsymbol \eta}$ correspond to joint moments of $E\{ \Tr \left[ A H_1 \right]^{i_1} \cdots \Tr \left[ A H_m \right]^{i_m} \}$ with $A$ given in (\ref{(noncentralterm)}).
Differently from \cite{Letac1, Letac2}, where the representation theory of symmetric group is resorted in order to compute the moments of non-central Wishart distributions, Theorem \ref{Wifund2}
involves integer partitions. Proposition \ref{17} shows the connection between the two methods.
\begin{prop} \label{17} For $\ibs=(1, \ldots, 1)$
\begin{eqnarray}
E[(n \punt \beta \punt \tilde{{\boldsymbol \rho}})^{\ibs}] & = &  \sum_{\sigma \in {\mathfrak S}_m} n^{|C(\sigma)|} \prod_{c \in C(\sigma)} \Tr \left( \prod_{j \in c} \Sigma H_j \right), \label{(centrWishp)}\\
E[(-1 \punt \beta \punt \tilde{{\boldsymbol \eta}})^{\ibs}]
& = & \sum_{\sigma \in {\mathfrak S}_m} (-1)^{|C(\sigma)|} \prod_{c \in C(\sigma)} l(c) \Tr \left(\Omega \prod_{j \in c} \Sigma H_j \right). \label{(formalAp)}
\end{eqnarray}
\end{prop}
\begin{proof}
Partitions of the multi-index $(1,\ldots,1)$ are matrices with entries equal to $0$ and $1.$ When $\ibs$ has entries equal to $0$ or $1,$ the multivariate moment of ${\boldsymbol \rho}$ involves
cyclic permutations of $(\Sigma H_1)^{i_1} \cdots (\Sigma H_k)^{i_k}$ with some powers equal to $0$ and the remaining equal to $1.$ The elements in the strings ${\mathfrak a}$ in (\ref{(firstpart)}) and (\ref{(secondpart)}) are all different and chosen in $\{1,\ldots,m\},$ so that ${\boldsymbol M}_m [\ibs] = {\boldsymbol N}_m [\ibs].$ Therefore for these
multi-indexes, equation (\ref{(firstpart)}) reduces to $E[{\boldsymbol \rho}^{\ibs}] = \Tr [(\Sigma H_1)^{i_1} \cdots (\Sigma H_k)^{i_k}].$
Equation (\ref{(centrWishp)}) follows by observing that for  $\ibs=(1, \ldots, 1)$ and from (\ref{(eq:16)})
$$E[(n \punt \beta \punt \tilde{{\boldsymbol \rho}})^{\ibs}] = \sum_{\pi \in \Pi_{|\ibs|}} {\mathfrak p}_{\pi} \, n^{|\pi|} \, \prod_{B \in \pi} \Tr \left[ \left( \prod_{j \in B} \Sigma H_j \right) \right],$$
with $\Pi_{|\ibs|}$ the set of all partitions \footnote{A partition $\pi$ of $\{\nu_1,  \ldots, \nu_i\}$ is a collection $\pi=\{B_1,  \ldots, B_k\}$
of $k \leq i$ disjoint non-empty subsets of $\{\nu_1, \ldots, \nu_i\}$ whose union is $\{\nu_1,  \ldots, \nu_i\}$. We denote the set of all partitions of the set $\{\nu_1, \ldots, \nu_i\}$ with $\Pi_i.$} of $\{1,2,\ldots,|\ibs|\}$
and ${\mathfrak p}_{\pi}  = \prod_{B \in \pi} (|B|-1)!$ the number of permutations of $\{1,2,\ldots,|\ibs|\}$ corresponding to the partition $\pi.$ The result follows by indexing the summation with permutations instead of partitions.
Equation (\ref{(formalAp)}) follows by similar arguments.
\end{proof}
For $\ibs=(1, \ldots, 1)$ and by using the permutation identity $e \in {\mathfrak S}_m,$ we have
$$E[(n \punt \beta \punt \tilde{{\boldsymbol \rho}})^{\ibs}] = E\left\{ \prod_{c \in C(e)} \Tr \left[ \widehat{W}(n) H_c \right] \right\} \quad \hbox{and} \quad E[(-1 \punt \beta \punt \tilde{{\boldsymbol \eta}})^{\ibs}]
=  E \left\{ \prod_{c \in C(e)} \Tr \left[ A H_c \right] \right\}.$$
By the group action of ${\mathfrak S}_m$ on ${\mathbb C},$ from (\ref{(centrWishp)}) and (\ref{(formalAp)}) we have
\begin{eqnarray}
E\left\{ \prod_{c \in C(\sigma)} \Tr \left[ \prod_{j \in c} \widehat{W}(n) H_j \right] \right\} & = & \sum_{\tau \in {\mathfrak S}_m} n^{l(\sigma \, \tau^{-1})} \prod_{c \in C(\tau)} \Tr \left( \prod_{j \in c} \Sigma H_j \right), \label{(genmom)} \\
E\left\{ \prod_{c \in C(\sigma)} \Tr \left[  \prod_{j \in c} A \, H_j \right] \right\} & = & \sum_{\tau \in {\mathfrak S}_m} (-1)^{l(\sigma \, \tau^{-1})} \prod_{c \in C(\tau)} l(c) \, \Tr \left(\Omega \prod_{j \in c} \Sigma H_j \right). \label{(genmom1)}
\end{eqnarray}
\subsection{Joint cumulants.}
Joint cumulants of non-central Wishart distributions can be recovered from Theorem \ref{Wifund2}, by using the additivity property given in Proposition 4.1 of \cite{faa}. Denote the $m$-tuple $(\Tr[W(n)H_1], \ldots,$ $\Tr[W(n)H_m])$
with $\mubs.$ Its $\ibs$-th joint cumulant $\hbox{Cum}_{\ibs}(\Tr[W(n)H_1], \ldots, \Tr[W(n)H_m])$  is the multivariate moment of $(\chi \, \punt \, \mubs)^{\ibs}.$ From (\ref{(jointmult1)}) we have
\begin{equation}
E[(\chi \, \punt \, \mubs)^{\ibs}] = E\left\{[\chi \, \punt \, (-1) \, \punt \, \beta \, \punt \, \tilde{\boldsymbol \eta}]^{\ibs} \right\} + E[(\chi \, \punt \, n \, \punt \, \beta \, \punt \, \tilde {\boldsymbol \rho})^{\ibs}].
\label{(cumwish)}
\end{equation}
Equation (\ref{(eq:16)}) allows us to compute $E[(\chi \punt (-1) \punt \, \beta \, \punt \tilde{\boldsymbol \eta})^{\ibs}]$ and $E[(\chi \, \punt \,n \, \punt \, \beta \, \punt \, \tilde {\boldsymbol \rho})^{\ibs}],$
with $\alpha$ replaced by $\chi \, \punt \, (-1)$ and $\chi \, \punt \, n$ respectively, and $\nubs$ replaced by $\tilde{\boldsymbol \eta}$ and $\tilde {\boldsymbol \rho}$ too. Since the moments of $\chi \, \punt \, (-1)$
and $\chi \, \punt \, n$ are all zero except the first, the only contribution in (\ref{(eq:16)}) is for $\lambdabs = \ibs$ for which $d_{\lambdabs} = 1.$ Then the following theorem is
proved, which generalizes Theorem \ref{cum2} and Proposition  \ref{cumwish}.
\begin{thm} \label{(multcum)}
$\hbox{\rm Cum}_{\ibs}(\Tr[W(n)H_1], \ldots, \Tr[W(n)H_m])  =  \ibs! (  n E[{\boldsymbol \rho}^{\ibs}] - E[{\boldsymbol \eta}^{\ibs}]),$ with $E[{\boldsymbol \rho}^{\ibs}]$ and
$E[{\boldsymbol \eta}^{\ibs}]$ given in  (\ref{(firstpart)}) and (\ref{(secondpart)}) respectively.
\end{thm}
\begin{example}
If  $m=2$ and $\ibs=(1,2),$ then
\begin{eqnarray*}
\hbox{\rm Cum}_{(1,2)}(\Tr[W(n)H_1],\Tr[W(n)H_2]) & = & 2! \left\{ n \Tr\left[ (\Sigma H_1) (\Sigma H_2)^2 \right] - \Tr \left[ \Omega (\Sigma H_1) (\Sigma H_2)^2 \right] \right.\\
& - & \left. \Tr \left[ \Omega (\Sigma H_2)^2  (\Sigma H_1) \right] - \Tr \left[ \Omega (\Sigma H_1) (\Sigma H_2) (\Sigma H_1) \right] \right\}.
\end{eqnarray*}
\end{example}
Depending on the symbolic representation of non-central Wishart distributions (\ref{(jointmult1bis)}), cumulants of the randomized version are obtained
from (\ref{(cumwish)}) by replacing  $\chi \punt n$ with  $\chi \punt \alpha,$ which is the $\alpha$-cumulant umbra.
Then Theorem \ref{(multcum)} needs to be updated as follows.
\begin{thm} \label{(multcumrand)}
If $\{c_i\}$ is the sequence of cumulants of $\alpha,$ then
\begin{equation}
\hbox{\rm Cum}_{\ibs}(\Tr[W(\alpha)H_1], \ldots, \Tr[W(\alpha)H_m]) =  \ibs! \left( \sum_{\lambdabs \mmodels \tbs}
\frac{c_{l(\lambdabs)}}{{\mathfrak m}(\lambdabs)} \prod_{\lambdabs_i} E[{\boldsymbol \rho}^{\lambdabs_i}]^{r_i} - E[{\boldsymbol \eta}^{\ibs}] \right).
\label{(cumrand)}
\end{equation}
\end{thm}
\begin{example}
If $m=2$ and $\ibs=(1,2),$ in order to compute $\hbox{\rm Cum}_{(1,2)}(\Tr[W(\alpha)H_1],\Tr[W(\alpha)H_2])$ and to better understand the powerful of the symbolic method, we could
employ the results of Example \ref{excum} for the part not involving the non-centrality matrix $\Omega.$ Indeed, the summation in (\ref{(cumrand)}) corresponds to the latter in (\ref{(addendi)})
with $n^{l(\lambdabs)}$ replaced by $c_{l(\lambdabs)}.$ Then in Example \ref{excum}, we can select the terms not involving $\Omega$ and replace the occurrences of
$n^k$ with $c_k.$ The result is the following:
\begin{eqnarray*}
&  &  \hbox{\rm Cum}_{(1,2)}(\Tr[W(\alpha)H_1],\Tr[W(\alpha)H_2]) =
4 \,c_{2} \Tr \left( H_{{1}} H_{{2}} \right) \Tr \left( H_{{2}} \right) + 2 \, c_{3} \Tr
 \left( H_{{1}} \right)  \left( \Tr  \left( H_{{2}} \right)
 \right) ^{2} \\
 & + &  2 \, c_1 \Tr  \left( H_{{1}} {H_{{2}}}^{2} \right) + c_2  \Tr  \left( H_{{1}} \right) \Tr  \left( {H_{{2}}}^{2}
 \right) - 2 \Tr \left[ \Omega \, (\Sigma H_1)\, (\Sigma H_2)^2 \right] - 2 \Tr \left[ \Omega\, (\Sigma H_2)^2 \, (\Sigma H_1) \right] \\
 & - & 2 \Tr \left[ \Omega \,(\Sigma H_1) \,(\Sigma H_2) \,(\Sigma H_1) \right].
\end{eqnarray*}
\end{example}
\section{Applications.}
\subsection{Permanents.}
Let $Y = (y_{ij}) \in  {\mathbb C}^{p \times p}.$  For $d \in {\mathbb C},$ the $d$-permanent ${\hbox{\rm per}}_{d} (Y) $ of the matrix $Y$ is the function
\begin{equation}
{\hbox{\rm per}}_{d} (Y) : = \sum_{\sigma \in {\mathfrak S}_p} d^{|C(\sigma)|} \prod_{j=1}^p y_{j, \sigma(j)}.
\label{(perm)}
\end{equation}
The determinant function corresponds to $d = -1$ while $d=1$ gives the classical permanent function.
The $d$-extension of the master theorem \cite{Lu} states that if $Z = \hbox{\rm diag}(z_1, \ldots, z_m)$ and
$T = (t_{i,j})$ is a $m \times m$ matrix such that\footnote{If $A$ is a $m \times m$ matrix, then $\| A \| := \max \{ |\theta_1|, \ldots, |\theta_m| \},$ where
$\theta_1, \ldots, \theta_m$ are the eigenvalues of $A.$} $\norm{ZT} < 1$ then ${\hbox{\rm per}}_{d} [T(\ibs)]$ is the $\ibs$-coefficient of $\det(I - Z T)^{-d},$
$${\hbox{\rm per}}_{d} [T(\ibs)] = \frac{\partial^{|\ibs|}}{\partial z_1^{i_1} \cdots \partial z_m^{i_m}} \left[ \det(I - Z T)^{-d} \right]_{\zbs = {\boldsymbol 0}} \quad
\hbox{with} \quad
T(\ibs) = \left( \begin{array}{ccc}
t_{i_1,i_1} & \ldots &  t_{i_1,i_m} \\
\vdots & \vdots & \vdots \\
t_{i_m,i_1} & \ldots &  t_{i_m,i_m}
\end{array} \right).$$
From (\ref{(laptrasfbis)}), $\det(I - Z T)^{-d}$ is the g.f. of a class of central Wishart distributions. Then the following result holds when $Z T$ can be decomposed as $\Sigma \, (H_1 z_1 + \cdots + H_m z_m)$
for suitable matrices $\Sigma$  and $H_1, \ldots, H_m.$
\begin{prop} \label{perm}
If there exists a symmetric matrix $\Sigma \in {\mathbb C}^{p \times p}$ and $H_1, \ldots, H_m \in {\mathbb C}^{p \times p}$ such that $Z T = \Sigma (H_1 z_1 + \cdots + H_kz_k),$
then ${\hbox{\rm per}}_{d} [T(\ibs)] = E[(d \, \punt \, \beta \, \punt \, \tilde{\boldsymbol \rho})^{\ibs}]$ with $\tilde{\boldsymbol \rho}$ given in (\ref{(firstpart)}).
\end{prop}
A generalization of Proposition \ref{perm} consists in replacing the complex $d$ with an umbra $\alpha$ umbrally representing $\{a_i\},$ that is
$${\hbox{\rm per}}_{\alpha} [T(\ibs)] : = \sum_{\sigma \in {\mathfrak S}_p} a_{|C(\sigma)|} \prod_{j=1}^p \left[T(\ibs)\right]_{j, \sigma(j)} = E[(\alpha \, \punt \,  \beta \, \punt \, \tilde{\boldsymbol \rho})^{\ibs}].$$

\subsection{Spectral polykays.}
Cumulants of Wishart distributions have been employed in approximating the density of sample correlation matrix \cite{Kollo} or within
wireless communications \cite{Tulino}. Moreover, due to Theorem \ref{cum2}  estimations of cumulants can add more information on the structure of the covariance
matrix and of the non-centrality matrix. A classical way to estimate cumulants and their products is by using $k$-statistics and polykays \cite{Bernoulli}. They rely
on power sum symmetric polynomials involving i.i.d. r.v.'s of random samples. A different way has been recently proposed in \cite{DMS} by introducing the notion of spectral sample.
Consider the matrix $H_{[m \times p]}$ obtained by deleting the last $p-m$ rows of a random unitary matrix $H$ uniformly distributed with respect
to Haar measure \cite{Speicher}. The set of eigenvalues $\ybs=(y_1, \ldots, y_m) \in \Real^m$ of the $m \times m$
Hermitian random matrix $Y = H_{[m \times p]} X H_{[p \times m]}^{\dag}$ is called a {\sl spectral sample} of size $m$ from $(x_1, \ldots, x_p) \in {\mathbb R}^p$
with $X = \hbox{diag}(x_1, \ldots, x_p).$ If $\lambda = (1^{r_1}, 2^{r_2}, \ldots)$ is a partition of some nonnegative integers \footnote{Here $\lambda$ runs through all partitions of length less or equal to $m$ \cite{MacDonald}.},
spectral polykays ${\mathfrak K}_{\lambda}(\boldsymbol{\ybs})$ are symmetric polynomials expressed in terms of $\Tr(Y^i),$ such that
\begin{equation}
E[{\mathfrak K}_{\lambda}(\boldsymbol{\ybs})] = \prod_{i=1}^{l(\lambda)} E(\Tr[{\mathfrak C}(m)]^{\lambda_i})
\label{(defspect)}
\end{equation}
with $\Tr[{\mathfrak C}(m)]$ given in (\ref{(clascum)}). For non-central Wishart distributions and from Theorem \ref{relcum}, spectral polykays (\ref{(defspect)})  are unbiased estimators of normalized cumulants
\begin{equation}
{\hbox {\rm Cum}}_{\lambda}(\Tr[W(m)])  =  \prod_{i=1}^{l(\lambda)} \left({\hbox {\rm Cum}}_{i}(\Tr[W(m)])\right)^{r_i} = p^{l(\lambda)} E[{\mathfrak K}_{\lambda}(W(m))]. \label{(spectral10)}
\end{equation}
For central Wishart distributions, which are unitarily invariant \footnote{A unitarily invariant central Wishart matrix is such that $\widehat{W}(n)=U {\mathit \Lambda} U^{\dag},$ with $U$ a Haar matrix independent of the diagonal matrix $\mathit \Lambda$ of the eigenvalues of $\widehat{W}(n),$ and the superscript ${\dag}$ denoting the conjugate transpose. The columns of $U$ are the eigenvectors of $\widehat{W}(n).$} the spectral sampling is a way to extract information from the set of its eigenvalues when $X=\mathit \Lambda.$ In this case,
the matrix $Y$ is a principal sub matrix \cite{Drton} of $\widehat{W}(n),$ and ${\mathfrak K}_{\lambda}(\boldsymbol{\ybs})$ are unbiased estimators of their normalized cumulants. These estimators
do not depend from the size of the sub matrix, since $E[{\mathfrak K}_{\lambda}(W(n)_{m \, m})|W(n)]={\mathfrak K}_{\lambda}(W(n)),$
for $m \leq n,$ where $W(n)_{m \, m}$ denotes the principal sub matrix involving the first $m$ rows of $W(n),$ cf. \cite{DMS}. The spectral estimators of mean, variance, asymmetry and kurtosis are given in the following:
\begin{eqnarray*}
\kappa_{(1)}   &=& \frac{S_1}{m}, \,\,  \kappa_{(2)} = \frac {m S_2-S_1^2}{ m \left( m^2-1 \right)}, \,\, \kappa_{(3)} = 2 \, \frac {2 \, S_1^{3} - 3 \, m \,S_1 \, S_2 + m^{2} S_3}{ m
\left( m^2-1 \right)  \left( m^2-4 \right)}\\
\kappa_{(4)} &=& 6 \, \frac{-5 S_1^4 + 10 m S_1^2 S_2 + (3 - 2 m^2) S_2^2 - (4 + 4 m^2) S_1 S_3 +
(m + m^3) S_4}{m^2 \, (m^2-1) \, (m^2-4) \, (m^2-9)}
\end{eqnarray*}
\subsection{Generalized moments and open problems.}
Similarly to the product (\ref{(genmom)}) involving central Wishart matrices, a generalization of joint moments (\ref{(joint)}) is given by
\begin{equation}
E\left\{ \prod_{c \in C(\sigma)} \Tr \left[ \prod_{j \in c} W(n) H_j \right] \right\}, \qquad \sigma \in {\mathfrak S}_k.
\label{(genmom2)}
\end{equation}
Differently from (\ref{(genmom)}), this computation could not be performed  by using the group action of ${\mathfrak S}_k$ on ${\mathbb C}.$
A different strategy consists in using the convolution  given in Theorem \ref{Wifund1}.
The first step is to transform (\ref{(genmom2)}) in a summation of joint moments
involving $\widehat{W}(n)$ and $A.$  This is the result of the following proposition.
\begin{prop}  
\begin{equation}
E \left[ \prod_{c \in C(\sigma)} \Tr \left( \prod_{j \in c} W(n) \, H_j \right) \right] = \sum_{(B_1, \ldots, B_k) \, \in \{\widehat{W}(n), \, A\}^k} \prod_{c \in C(\sigma)} E \left[\Tr \left( \prod_{j \in c} B_j H_j \right) \right].
\label{(firststepeq)}
\end{equation}
\end{prop}
\begin{proof}
In (\ref{(genmom2)}), replace $W(n)$ with $\widehat{W}(n) + A$ and observe that
$$\prod_{c \in C(\sigma)} \Tr \left( \prod_{j \in c} (\widehat{W}(n) + A) H_j \right) = \prod_{c \in C(\sigma)} \sum_{(B_{j_1}, \ldots, B_{j_{l(c)}}) \, \in \{\widehat{W}(n), \, A\}^k} \Tr (B_{j_1} H_{j_1} \cdots
B_{j_{l(c)}} H_{j_{l(c)}}).$$
The result follows by multiplying the summations and by applying the operator $E.$
\end{proof}
The second step is to compute the right-hand-side of (\ref{(firststepeq)}) by separating the contribution of $\widehat{W}(n)$ and $A,$ in order to employ (\ref{(genmom)}) and (\ref{(genmom1)}).
A way to perform such computation is to resort free probability \cite{Speicher}, since $\widehat{W}(n)$ and $A$ are elements of a noncommutative probability space\footnote{A noncommutative probability space is a pair $({\mathbb A},\Phi)$, where ${\mathbb A}$ is a unital noncommutative algebra and $\Phi: {\mathbb A} \rightarrow {\mathbb C}$ is a unital linear functional. } $(M_n({\mathbb C}[{\mathcal A}]),\Tr)$, with $M_n({\mathbb C}[{\mathcal A}])$ the algebra of umbral matrices with the usual matrix multiplication. In the following, a way involving free cumulants is suggested to perform such computation: it would be convenient to characterize some closed form formula involving the functional $\Tr$ without the employment of free cumulants. This is still an open problem.

Let ${\mathcal NC}$ be the lattice of all noncrossing partitions\footnote{A noncrossing partition $\pi = \{B_1,B_2, \ldots ,B_k\}$ of the set $[i]$ is a partition such that if $1 \leq h < l < s < k \leq i$, with $h, s \in B_n$ and $l,k \in B_{n^{\, \prime}}$, then $n = n^{\prime}$. } of $[i]$. For any noncrossing partition $\pi,$ set
$$\Tr_{\pi}(B_1, \ldots, B_i)=\prod_{D \in \pi} \Tr (B_{j_1} \cdots B_{j_s})$$
for $D=(j_1 < \cdots < j_s)$. Free cumulants are defined as multilinear functionals such that
\begin{equation}
c_{\pi}(B_1, \ldots, B_i)=\prod_{D \in \pi} c_{|D|}(B_{j_1} \cdots B_{j_s}) \quad \hbox{with} \quad c_i(B_1, \ldots, B_i)=\sum_{\pi \in {\mathcal NC}} {\mathfrak m}(\pi, 1_i) \, \Tr_{\pi}(B_1, \ldots, B_i),
\label{(freecum)}
\end{equation}
where ${\mathfrak m}(\pi, 1_i)$ is the Moebius function on the lattice of noncrossing partitions \cite{Speicher}.  For two sets $\{S_1, \ldots, S_k\}$ and $\{T_1, \ldots, T_k\}$ of random matrices satisfying
the hypothesis of Theorem 14.4 in \cite{Speicher}, we have
\begin{equation}
\Tr(S_1 T_1 \cdots S_k T_k) = \sum_{\pi_S \in {\mathcal NC}(1,3,\ldots,2 n - 1)} c_{\pi_S}[S_1, \ldots, S_k] \left( \sum_{\pi_T \in {\mathcal NC}(2,4,\ldots,2 n) \atop \pi_S \cup \pi_T \in {\mathcal NC}(n)}
c_{\pi_T}[T_1, \ldots, T_k]\right).
\label{(tool)}
\end{equation}
The products on the right-hand-side of (\ref{(firststepeq)}) could be computed by using (\ref{(tool)}) with a suitable choice of the matrices $S_i$ and $T_i.$ For example, in order to compute $\Tr(A \, H_1
\, A \, H_2 \, \widehat{W}(n) \, H_3 \, \widehat{W}(n) \, H_4)$ choose $S_1 = A \, H_1, T_1 = I, S_2 = A \, H_2, T_2 = \widehat{W}(n) \, H_3, S_3 = I, T_3 = \widehat{W}(n) \, H_4.$
An expression of (\ref{(firststepeq)}) in terms of (\ref{(genmom)}) and (\ref{(genmom1)}) follows by using (\ref{(freecum)}) and by characterizing the permutations $\sigma$ corresponding to the non-crossing partitions $\pi$
involved in (\ref{(tool)}). The hypothesis of Theorem 14.4 in \cite{Speicher} rely on the notion of freely independence of $S_i$ and $T_i.$ This theorem could be still employed if  $S_i$ and $T_i$
are asymptotical free independent, that is if $A$ and $\widehat{W}(n)$ are asymptotical free independent. This is true if the fluctuation of $A$ is not too large, behaving
like a constant matrix. Of course, this fluctuation depends on the non-centrality matrix $\Omega.$ So, if $A$ admits an eigenvalue distribution and all normalized traces of $A$ would have $1/p^2$ estimates \cite{Speicher},
then $A$ and $\widehat{W}(n)$ are asymptotical free independent.

\end{document}